\documentclass[11pt]{article}
\usepackage{amsfonts}
\usepackage{graphics}
\usepackage{indentfirst}
\usepackage{color}
\usepackage{cite}
\usepackage{latexsym}
\usepackage[paper=a4paper, left=1.6cm, right=1.6cm, top=1.8cm, bottom=1.6cm, headheight=5.5pt, footskip=0.8cm, footnotesep=0.8cm, centering, includefoot]{geometry}
\usepackage{amsmath}
\allowdisplaybreaks
\usepackage{amssymb}
\usepackage[colorlinks, linkcolor=red]{hyperref}
\hypersetup{urlcolor=red, citecolor=red}
\usepackage[dvips]{epsfig}
\usepackage{amscd}

\usepackage{amsthm}
\usepackage{mathrsfs}
\usepackage{verbatim}

\DeclareMathOperator{\divv}{div}
\DeclareMathOperator{\curl}{curl}
\DeclareMathOperator{\loc}{loc}
\DeclareMathOperator{\divf}
{di\overset{\raisebox{0.1ex}{\hspace{0.1em}$\mathbf{\cdot}$}}{v}}
\allowdisplaybreaks

\begin{document}
\title{Global weak solutions and incompressible limit to the isentropic compressible Navier--Stokes equations in the half-plane with ripped density and large initial data
\thanks{
Wu's research was partially supported by Fujian Alliance of Mathematics (No. 2023SXLMMS08) and the Scientific Research Funds of Xiamen University of Technology (No. YKJ25009R).
Zhong's research was partially supported by Fundamental Research Funds for the Central Universities (No. SWU--KU24001) and National Natural Science Foundation of China (No. 12371227). }
}

\author{Shuai Wang$\,^{\rm 1}\,$,\ Guochun Wu$\,^{\rm 2}\,$,\
Xin Zhong$\,^{\rm 1}\,$ {\thanks{E-mail addresses: swang238@163.com (S. Wang),
guochunwu@126.com (G. Wu), xzhong1014@amss.ac.cn (X. Zhong).}}\date{}\\
\footnotesize $^{\rm 1}\,$
School of Mathematics and Statistics, Southwest University, Chongqing 400715, P. R. China\\
\footnotesize $^{\rm 2}\,$ School of Mathematics and Statistics, Xiamen University of Technology, Xiamen 361024, P. R. China}

\maketitle
\newtheorem{theorem}{Theorem}[section]
\newtheorem{definition}{Definition}[section]
\newtheorem{lemma}{Lemma}[section]
\newtheorem{proposition}{Proposition}[section]
\newtheorem{corollary}{Corollary}[section]
\newtheorem{remark}{Remark}[section]
\renewcommand{\theequation}{\thesection.\arabic{equation}}
\catcode`@=11 \@addtoreset{equation}{section} \catcode`@=12
\maketitle{}

\begin{abstract}
We prove the global existence of weak solutions to the isentropic compressible Navier--Stokes equations with ripped density in the half-plane under a slip boundary condition provided the bulk viscosity coefficient is properly large. Moreover, we show that such solutions converge globally in time to a weak solution of the inhomogeneous incompressible Navier--Stokes equations as the bulk viscosity coefficient tends to infinity. In particular, {\it the large initial data and an initial patch of density as well as a vacuum are allowed}. Our method relies on a Desjardins-type logarithmic interpolation inequality and some new techniques based on the effective viscous flux.
\end{abstract}




\tableofcontents

\section{Introduction}
\subsection{Background and motivation}
This paper deals with the classical isentropic compressible  Navier--Stokes equations in the half-plane $\mathbb{R}^2_+=\{\mathbf{x}\in\mathbb{R}^2:x_2>0\}$:
\begin{align}\label{a1}
\begin{cases}
\rho_t+\divv(\rho\mathbf{u})=0,\\
(\rho\mathbf{u})_t+\divv(\rho\mathbf{u}\otimes\mathbf{u})-\mu\Delta\mathbf{u}
-(\mu+\lambda)\nabla\divv\mathbf{u}+\nabla P
=0.
\end{cases}
\end{align}
The unknowns $\rho$, $\mathbf{u}=(u^1,u^2)$, and $P=P(\rho)=a\rho^\gamma\ (a>0,\gamma>1)$ are the fluid density, velocity, and pressure, respectively. The constants $\mu$ and $\lambda$ represent the shear viscosity and bulk viscosity of the fluid, respectively, satisfying physical restrictions
\begin{equation*}
\mu>0,\ \ \ \mu+\lambda\geq0.
\end{equation*}
The system \eqref{a1} is supplemented with initial and boundary conditions
\begin{equation}
(\rho,\mathbf{u})|_{t=0}=(\rho_0,\mathbf{u}_0)(\mathbf{x}),\ \ \mathbf{x}\in\mathbb{R}^2_+,
\end{equation}
\begin{equation}\label{a3}
(u^1,u^2)=(u^1_{x_2},0),\ \ x_2=0,\ t>0,
\end{equation}
and the far-field constraint
\begin{align}\label{a4}
(\rho_0,\mathbf{u}_0)(\mathbf{x})\rightarrow(\tilde{\rho},\mathbf{0}) \ \text{as}\ |\mathbf{x}|\rightarrow\infty,\ \ \mathbf{x}\in\mathbb{R}^2_+,
\end{align}
where the constant $\tilde{\rho}>0$ is the reference density.

For $t\geq0$, observe that solutions of \eqref{a1} obey the global energy law
\begin{align*}
\int_{\mathbb{R}^2_+}\bigg(\frac{1}{2}\rho |\mathbf{u}|^2+G(\rho)\bigg)\mathrm{d}\mathbf{x}+\int_0^t\left[\mu\|\nabla \mathbf{u}\|_{L^2}^2+(\mu+\lambda)\|\divv \mathbf{u}\|_{L^2}^2\right]\mathrm{d}\tau
+\int_0^t\int_{\partial\mathbb{R}^2_+}\mu|\mathbf{u}|^2\mathrm{ds}\mathrm{d}\tau
\leq C_0,
\end{align*}
where the initial total energy $C_0$ and the potential energy $G(\rho)$ are given by
\begin{equation}\label{z1.6}
C_0\triangleq\int_{\mathbb{R}^2_+}
\bigg(\frac{1}{2}\rho_0|\mathbf{u}_0|^2+G(\rho_0)\bigg)\mathrm{d}\mathbf{x},\ \
G(\rho)\triangleq\rho\int_{\tilde{\rho}}^\rho\frac{P(v)-P(\tilde{\rho})}{v^2}\mathrm{d}v.
\end{equation}
Thus, the incompressibility can be recovered formally from the global energy law as $\lambda$ tends to infinity.

The Navier-slip boundary condition is defined pointwise as the flow velocity being parallel to the tangential component of the stress on the boundary. For the flat boundary, a more explicit expression in \cite{Hoff05} for the half-space $\mathbb{R}^N_+=\big\{\mathbf{x}\in\mathbb{R}^N:x_N>0\big\}$ shows that
\begin{equation*}
\mathbf{u}(\mathbf{x})=-\mu K(\mathbf{x})\boldsymbol\omega\mathbf{n}
\end{equation*}
on $\partial\mathbb{R}^N_+=\big\{\mathbf{x}\in\mathbb{R}^N:x_N=0\big\}$, where $K$ is a positive smooth function, $\boldsymbol\omega$ is the vorticity matrix with $\omega^{j,k}=u^j_{x_k}-u^k_{x_j}$, and $\mathbf{n}=-\mathbf{e}_N$ is the unit outer normal vector. Without
loss of generality, we take $\mu K\equiv1$ in this paper, which is equivalent to the boundary condition \eqref{a3}.

System \eqref{a1} can be derived from the laws of classical physics: the first equation is the mass balance, while the second equation corresponds to the momentum balance in the absence of external forces. The reader may refer to \cite[Chapter 1]{PL96} for more details. Due to its importance in both mathematical and practical application, the global well-posedness for multi-dimensional isentropic compressible Navier--Stokes equations has received special attention over the past decades. It is actually not possible to mention all the remarkable and interesting results in this regards. Instead we will limit ourselves to citing a few of them. In \cite{MN80,MN83}, Matsumura and Nishida established global classical solutions for three-dimensional (3D) initial-boundary value problem when the initial data are a small perturbation of constant equilibrium states in $H^3$. Global existence of strong solutions for small initial data $(\rho_0,\mathbf{u}_0)$ in critical Besov spaces $\dot{B}^{\frac{N}{2}}_{2,1}(\mathbb{R}^N)
\times\dot{B}^{\frac{N}{2}-1}_{2,1}(\mathbb{R}^N)\ (N\geq2)$ was proven by Danchin \cite{Da00}. This significant result was generalized in \cite{CD10,CCZ10,H11} for the critical $L^p$ framework. A further progress in \cite{FZZ18,HHW19,ZLZ20} shows global solutions for certain classes of large initial data. It should be noted that all the results stated above require the initial density to be strictly positive, that is, exclude the possible appearance of vacuum states.

However, as pointed out by many authors (see, e.g., \cite{HLX12,M1,M2}), the mathematical analysis of compressible fluid near vacuum regions presents significant challenges due to the inherent singularity and degeneracy of the governing equations. The major breakthrough for isentropic compressible  Navier--Stokes equations with vacuum was made by P.-L. Lions \cite{PL98}, who developed renormalization techniques and the effective viscous flux arguments to construct global weak solutions in $\mathbb{R}^N$ for $\gamma\geq\frac{3N}{N+2}\ (N=2,3)$. Since then some important developments have been achieved by several authors. Feireisl and collaborators in \cite{F04,EF01} extended Lions' result to $\gamma>\frac{N}{2}$ with the aid of oscillation defect measures. For initial data with spherical symmetry and axisymmetry, respectively, Jiang and Zhang \cite{JZ01,JZ03} derived global weak solutions for any $\gamma>1$. Their approach relies crucially on exploiting the properties of the effective viscous flux to enhance integrability estimates on the density. Meanwhile, Bresch and Jabin \cite{BJ18} investigated the global existence of weak solutions to the compressible Navier--Stokes equations with non-monotone pressure and anisotropic stress tensor. More recently, a new construction of global weak solutions to compressible Navier--Stokes equations with pressure $P(\rho)=\rho^\gamma$ for $\gamma\geq\frac95$ in three dimensions was obtained in \cite{CMZ24}. Despite these advances, the regularity and uniqueness for such weak solutions remain unknown.

There are also very interesting investigations concerning the global well-posedness of strong (or classical) solutions for isentropic compressible  Navier--Stokes equations with vacuum. A significant contribution was made in \cite{HLX12}, where the global existence of a unique classical solution in $\mathbb{R}^3$ for initial data possessing small total energy but potentially large oscillations and permitting far field vacuum was demonstrated. Later on, the 3D Cauchy problem in \cite{HLX12} was extended to the whole plane case with vacuum at infinity \cite{LX19} via weighted energy estimates and decay rate analysis. More recently, the method developed in \cite{HLX12} was generalized to 3D initial-boundary value problem with Navier-slip boundary conditions in \cite{CL23} and 3D Cauchy problem with large initial energy when the adiabatic exponent $\gamma$ approaches $1$ in \cite{HHPZ24}. It should be noticed that a central feature of their studies in \cite{HLX12,LX19} is the derivation of both the time-independent upper bound of the density and time-dependent higher-order estimates of smooth solutions.

Apart from {\it large-energy weak solutions} \cite{PL98,EF01} and {\it small-energy classical solutions} \cite{HLX12,LX19}, the third type of solutions to \eqref{a1} is the so-called {\it Hoff's intermediate weak solutions}. More precisely, by imposing specific structural conditions on the initial data, Hoff \cite{Hoff95,Hoff95*} proved the global existence of such solutions in the whole space.
Indeed, the regularity of such weak solutions is stronger than Lions--Feireisl's solutions in the sense that particle path can be defined in the non-vacuum region, but weaker than the usual strong solutions in the sense that discontinuity of the density can be transported along the particle path (cf. \cite{Hoff02}).
More recently, based on some new techniques involving the effective viscous flux, Hu--Wu--Zhong \cite{HWZ} established the global existence of Hoff's weak solutions to the Cauchy problem in $\mathbb{R}^3$ with bounded nonnegative initial density that are of small energy but possibly large oscillations. Moreover, these solutions converge globally in time to a global weak solution of the inhomogeneous incompressible Navier--Stokes equations as the bulk viscosity goes to infinity.

Meanwhile, when the bulk viscosity is sufficiently large, Danchin and Mucha \cite{DM17} obtained global regular solutions to the isentropic compressible Navier--Stokes equations in $\mathbb{R}^2$ with initial density bounded away from vacuum and arbitrarily large initial velocity. As a by-product, they derived the incompressible limit as $\lambda\rightarrow\infty$ with a convergence rate of order $(2\mu+\lambda)^{-1/2}$. Subsequently, Danchin and Mucha \cite{DM23} generalized the result in \cite{DM17} to the torus $\mathbb{T}^2$, allowing vacuum states but assuming zero initial total momentum. It is certainly interesting to investigate whether these results can be extended to the case with physical boundaries and the presence of  vacuum as well as non-zero initial total momentum. For more studies on the compressible Navier--Stokes equations, one may consult the excellent handbook \cite{GN18} and references contained therein.

It should be noted that the works mentioned above primarily focus on the Cauchy problem. As a typical case, the study of the isentropic compressible Navier--Stokes equations in the half-space with physical boundary conditions is also important mathematically and physically. The large-time behavior of solutions in the half-space with Dirichlet boundary conditions was studied in \cite{KK02,KK05}. The global classical and weak solutions with vacuum in the half-space $\mathbb{R}_+^3$ involving Navier boundary conditions were proven in \cite{D12,Hoff05}. Later on, Perepelitsa \cite{P14} generalized the result \cite{Hoff05} to the case of Dirichlet boundary conditions. However, the smallness assumptions on the initial energy are required in \cite{Hoff05,D12,P14}. Naturally, one may wonder whether the global solutions with {\it large initial data} in the half-space can be established. Motivated by \cite{DM17,DM23,HWZ}, the main purpose of this paper is to give an affirmative answer for these questions. More precisely, we shall establish global weak solutions to the initial-boundary value problem \eqref{a1}--\eqref{a4} with ripped density and large initial data, allowing the presence of a vacuum and non-zero initial total momentum. Furthermore, we also show the limiting behavior of such solutions as $\lambda\rightarrow\infty$ to a solution of the inhomogeneous incompressible Navier--Stokes equations
\begin{align}\label{a5}
\begin{cases}
\varrho_t+{\bf v}\cdot\nabla\varrho=0,\\
\varrho{\bf v}_t+\varrho{\bf v}\cdot\nabla {\bf v}+\nabla \Pi-\mu\Delta {\bf v}=0,\\
\divv{\bf v}=0,\\
(\varrho,{\bf v})\big |_{t=0}=(\rho_0,{\bf v}_0),
\end{cases}
\end{align}
for $({\bf x},t)\in\mathbb{R}^{2}_+\times(0,+\infty)$. Here ${\bf v}_0$ is the Leray-Helmholtz projection of ${\bf u}_0$ on divergence-free vector fields.

\subsection{Main results}
Before stating our main results, we first formulate the notations and conventions used throughout this paper. We denote by $C$ a generic positive constant which may vary at different places. The symbol $\Box$ denotes the end of a proof and $a\triangleq b$ means $a=b$ by definition. For $1\le p\le \infty$ and integer $k\ge 0$, we denote the standard Sobolev spaces as follows:
\begin{align*}
\begin{cases}
L^p=L^p(\mathbb{R}^2_+),\ \ W^{k, p}=W^{k, p}(\mathbb{R}^2_+),\ \ H^k=W^{k, 2}, \\
D^{1,p}=\{f\in L_{\loc}^1(\mathbb{R}^2_+):\|\nabla f\|_{L^p}<\infty\},\ \
\widetilde{H}^1=\big\{\mathbf{f}\in H^1:(\mathbf{f}\cdot\mathbf{n})|_{\partial\mathbb{R}^2_+}=0\big\}.
\end{cases}
\end{align*}
In addition, for any $f\in L^1_{\loc}(\mathbb R^2_+)$, we define its mollification by $[f]_\epsilon\triangleq j_\epsilon*f$, where $j_\epsilon=j_\epsilon({\bf x})$ is the standard mollifier with width $\epsilon$.
For two $n\times n$ matrices $A=\{a_{ij}\}$ and $B=\{b_{ij}\}$, the symbol $A: B$ represents the trace
of $AB$, that is,
\begin{equation*}
A:B\triangleq\operatorname{tr}(AB)=\sum_{i,j=1}^na_{ij}b_{ji}.
\end{equation*}
For $\alpha\in (0,1]$, the H\" older seminorm of a function ${\bf v}:U\subseteq\overline{\mathbb{R}^2_+}\rightarrow \mathbb R^2$ is defined by
\begin{align*}
\langle {\bf v}\rangle^\alpha_{U}
=\sup\limits_{\substack{{\bf x}, {\bf y}\in U\\ {\bf x}\neq {\bf y}}}
\frac{|{\bf v}({\bf x})-{\bf v}({\bf y})|}{|{\bf x}-{\bf y}|^\alpha}.
\end{align*}

Furthermore, we denote by
\begin{align}\label{z1.5}
\begin{cases}
\dot{f}\triangleq f_t+\mathbf{u}\cdot\nabla f,\\
F\triangleq(2\mu+\lambda)\divv\mathbf{u}-(P(\rho)-P(\tilde{\rho})),\\
\omega\triangleq\curl\mathbf{u}=-\nabla^\bot\cdot\mathbf{u}=-\partial_2u^1+\partial_1u^2,
\end{cases}
\end{align}
which express the material derivative of $f$, the effective viscous flux, and the vorticity, respectively. For simplicity, we write
\begin{align*}
\int f \mathrm{d}\mathbf{x}=\int_{\mathbb{R}^2_+} f \mathrm{d}\mathbf{x}, \ \ f_i=\partial_if=\frac{\partial f}{\partial x_i}.
\end{align*}
Moreover, we introduce the Helmholtz projection
\begin{equation*}
\mathcal{P}\triangleq \text{Id}+\nabla(-\Delta)^{-1}\divv
\end{equation*}
onto the subspace of divergence-free vector fields $\mathbf{u}$ with the no-penetration boundary condition and define $\mathcal{Q}\triangleq \text{Id}-\mathcal{P}$, where both operators map $L^p$ into itself for any $1<p<\infty$.

We recall the definition of weak solutions to the problem \eqref{a1}--\eqref{a4} in the sense of \cite{Hoff05,NS04}.
\begin{definition}\label{d1.1}
A pair $(\rho, \mathbf{u})$ is said to be a weak solution to the problem \eqref{a1}--\eqref{a4} provided that
\begin{equation*}
  \rho-\tilde{\rho}\in C([0,\infty);H^{-1}(\mathbb{R}^2_+)),\ \
  \rho\mathbf{u}\in C([0,\infty);\widetilde{H}^{1}(\mathbb{R}^2_+)^*),\ \
  \nabla\mathbf{u}\in L^2(\mathbb{R}^2_+\times(0,\infty))
\end{equation*}
with $(\rho,\mathbf{u})|_{t=0}=(\rho_0,\mathbf{u}_0)$, where $\widetilde{H}^{1}(\mathbb{R}^2_+)^*$ is the dual of $\widetilde{H}^{1}(\mathbb{R}^2_+)$.
 Moreover,
for any $t_2\geq t_1\geq 0$ and any test function $(\phi,\boldsymbol\psi)(\mathbf{x},t)\in C^1\big(\overline{\mathbb{R}^2_+}\times[t_1,t_2]\big)$, with uniformly bounded
support in $\mathbf{x}$ for $t\in[t_1,t_2]$ and satisfying $(\boldsymbol\psi\cdot\mathbf{n})|_{\partial\mathbb{R}^2_+}=0$, the following identities hold\footnote{Throughout this paper, we will use the Einstein summation over repeated indices convention.}:
\begin{gather*}
\int_{\mathbb{R}^2_+}\rho(\mathbf{x},\cdot)\phi(\mathbf{x},\cdot)
\mathrm{d}\mathbf{x}\Big|_{t_1}^{t_2}=\int_{t_1}^{t_2}\int_{\mathbb{R}^2_+}(\rho\phi_t+
\rho\mathbf{u}\cdot\nabla\phi)\mathrm{d}\mathbf{x}\mathrm{d}t,\label{z1.7}\\
\int_{\mathbb{R}^2_+}(\rho\mathbf{u}\cdot\boldsymbol{\psi})
(\mathbf{x},\cdot)\mathrm{d}\mathbf{x}\Big|_{t_1}^{t_2}=\int_{t_1}^{t_2}
\int_{\mathbb{R}^2_+}\big(\rho\mathbf{u}\cdot\boldsymbol{\psi}_t+
\rho u^i\mathbf{u}\cdot\partial_i\boldsymbol{\psi}+P\divv\boldsymbol{\psi}\big)\mathrm{d}\mathbf{x}\mathrm{d}t\notag\\
\quad-\int_{t_1}^{t_2}
\int_{\mathbb{R}^2_+}\big(\mu\partial_i\mathbf{u}\cdot\partial_i\boldsymbol{\psi}+(\mu+\lambda)\divv\mathbf{u}\divv\boldsymbol{\psi}\big)\mathrm{d}\mathbf{x}\mathrm{d}t
-\int_{t_1}^{t_2}
\int_{\partial\mathbb{R}^2_+}\mu\mathbf{u}\cdot\boldsymbol{\psi}\mathrm{ds}\mathrm{d}t.
\end{gather*}
\end{definition}
For the initial data $(\rho_0,\mathbf{u}_0)$, assume that there exist two positive constants $\hat{\rho}$ and $M$ (not necessarily small) such that
\begin{gather}
0\leq\inf\rho_0\leq\sup\rho_0\leq\hat{\rho},
 \label{z1.8}\\
\mathbf{u}_0\in \widetilde{H}^1,\  C_0+\mu\|\nabla\mathbf{u}_0\|_{L^2}^2+(\mu+\lambda)\|\divv\mathbf{u}_0\|_{L^2}^2\leq M.\label{z1.9}
\end{gather}
From \eqref{z1.6}, one can see that there exists a positive constant $C$ depending only on $\tilde{\rho}$ and $\hat{\rho}$ such that
\begin{equation}\label{z1.10}
  \frac{1}{C(\tilde{\rho},\hat{\rho})}(\rho-\tilde{\rho})^2\leq G(\rho)\leq C(\tilde{\rho},\hat{\rho})(\rho-\tilde{\rho})^2,
\end{equation}
where and in what follows we sometimes use $C(f)$ to emphasize the dependence on $f$.

Now we state our first result concerning the global existence of weak solutions.

\begin{theorem}\label{t1.1}
Let \eqref{z1.8} and \eqref{z1.9} be satisfied, there exists a positive number $D$ depending only
on $\tilde{\rho}$, $\hat{\rho}$, $a$, $\gamma$, and $\mu$ such that if
\begin{equation}\label{lam}
\lambda\geq\exp\bigg\{(2+M)^{e^{D(1+C_0)^5}}\bigg\},
\end{equation}
then the problem \eqref{a1}--\eqref{a4} admits a global weak solution $(\rho,\mathbf{u})$ in the sense of Definition $\ref{d1.1}$ satisfying
\begin{equation}\label{reg}
\begin{cases}0\leq\rho(\mathbf{x},t)\leq2\hat{\rho}~a.e.~\mathrm{on}~\mathbb{R}^2_+\times[0,\infty),\\
(\rho-\tilde{\rho},\sqrt{\rho}\mathbf{u})\in C([0,\infty);L^2(\mathbb{R}^2_+)),~\mathbf{u}\in L^\infty(0,\infty;H^1(\mathbb{R}^2_+)),\\
(\nabla^2\mathcal{P}\mathbf{u},\nabla F,\sqrt{\rho}\dot{\mathbf{u}})\in L^2(\mathbb{R}^2_+\times(0,\infty)),\\
\sigma^{\frac{1}{2}}\sqrt{\rho}\dot{\mathbf{u}}\in L^\infty(0,\infty;L^2(\mathbb{R}^2_+)),~ \sigma^{\frac{1}{2}}\nabla\dot{\mathbf{u}}\in L^2(\mathbb{R}^2_+\times(0,\infty)),
\end{cases}
\end{equation}
where $\sigma\triangleq\min\{1,t\}$.
\end{theorem}

The next result will treat the incompressible limit (characterised by the large value of the bulk viscosity) of the global weak solutions established in Theorem \ref{t1.1}.

\begin{theorem}\label{t1.2}
Let $\{(\rho^\lambda,{\bf u}^\lambda)({\bf x},t)\}$ be the family of solutions obtained in Theorem \ref{t1.1}. Then, there exists a subsequence $\{\lambda_k\}$
with $\lambda_k\rightarrow\infty$ such that
\begin{align}\label{1.13}
 \rho^{\lambda_{k}}\rightarrow \varrho ~~&\text{strongly in} ~ L^2(K), \ \
 \text{for any compact set} \ K \subset \mathbb{R}^2_+ \text{ and any} \ t\ge 0 ,\\
 {\bf u}^{\lambda_{k}}\rightarrow {\bf v}
 ~~&\text{uniformly on compact sets in}~\mathbb R^2_+\times(0,\infty),\notag
\end{align}
where $(\varrho,{\bf v})$ is a global weak solution to the inhomogeneous incompressible Navier--Stokes equations \eqref{a5} in the sense of Definition \ref{d1.2} below.
\end{theorem}

\begin{definition}\label{d1.2}
A pair $(\varrho, {\bf v})$ is said to be a weak solution to the problem \eqref{a5} provided that
\begin{gather}
\varrho\in L^\infty(\mathbb R^2_+\times (0,\infty)),~~
\sqrt{\varrho}\mathbf{v}\in L^\infty([0,\infty); L^2(\mathbb{R}^2_+)),~~
\nabla{\bf v}\in L^2(\mathbb R^2_+\times (0,\infty)),\notag\\
(\varrho-\tilde{\rho})\in C([0,\infty);L^{2}(\mathbb R^2_+)),\label{1.14}\\
 {\bf v}\in L^2(\mathbb R^2_+\times (0,T)), ~~ \text{for any} \ T>0.\label{1.15}
\end{gather}
Moreover, for any $t_2\geq t_1\geq 0$ and any $C^1$ test function $(\phi,\boldsymbol\psi)$
just as in Definition \ref{d1.1}, which additionally satisfies $\divv\boldsymbol\psi(\cdot,t)=0$ on $\mathbb R^2_+\times[0,\infty)$, the following identities hold:
\begin{gather}
\int_{\mathbb{R}^2_+}\varrho(\mathbf{x},\cdot)\phi(\mathbf{x},\cdot)
\mathrm{d}\mathbf{x}\Big|_{t_1}^{t_2}=\int_{t_1}^{t_2}\int_{\mathbb{R}^2_+}(\varrho\phi_t+
\varrho\mathbf{v}\cdot\nabla\phi)\mathrm{d}\mathbf{x}\mathrm{d}t,\label{1.16}\\
\int_{\mathbb{R}^2_+}(\varrho\mathbf{v}\cdot\boldsymbol{\psi})
(\mathbf{x},\cdot)\mathrm{d}\mathbf{x}\Big|_{t_1}^{t_2}=\int_{t_1}^{t_2}
\int_{\mathbb{R}^2_+}\big(\varrho\mathbf{v}\cdot\boldsymbol{\psi}_t+
\varrho v^i\mathbf{v}\cdot\partial_i\boldsymbol{\psi}-\mu\partial_i\mathbf{v}\cdot\partial_i\boldsymbol{\psi}\big)\mathrm{d}\mathbf{x}\mathrm{d}t
-\int_{t_1}^{t_2}
\int_{\partial\mathbb{R}^2_+}\mu\mathbf{v}\cdot\boldsymbol{\psi}\mathrm{ds}\mathrm{d}t.\label{1.17}
\end{gather}
\end{definition}

Several remarks are in order.

\begin{remark}
It should be noted that Theorem \ref{t1.1} holds for arbitrarily large initial energy as long as the bulk viscosity coefficient is suitably large, which is in sharp contrast to \cite{Hoff05,D12,P14} where the smallness condition on the initial energy is needed in order to obtain global solutions to the compressible Navier--Stokes equations in the half-space.
\end{remark}

\begin{remark}
Our work extends the significant results of Danchin and Mucha in \cite{DM17,DM23} to the half-plane, a domain with physical boundary to which their methods seem not to be directly applicable. Moreover, compared with \cite[Theorem 1.1 and Corollary 1.1]{DM17}, the density in our theorems is allowed to have large oscillation and vacuum states in interior regions. The assumption of zero initial total momentum in \cite[Theorems 2.1 and 2.3]{DM23} is also removed in our setting.
\end{remark}

\begin{remark}\label{r1.4}
The half-plane can be viewed as an intermediate case between the whole plane and bounded domains, where the emergence of the boundary introduces significant and unforeseen difficulties in applying the Hodge-type decomposition. Indeed, the intractable boundary integrals from the self-convection of $\mathcal{P}\mathbf{u}$ necessitate exploiting the boundary conditions and geometric flatness, and even extracting information from them concerning the bulk viscosity.
\end{remark}

\begin{remark}
As pointed out by Cramer \cite{CR}:``\textit{Several fluids, including common diatomic gases, are seen to have bulk viscosities which are hundreds or thousands of times larger than their shear viscosities}." That is, the large bulk viscosity coefficient is suitable for some physical models.
\end{remark}


\subsection{Strategy of the proof}

Now let us discuss the key issues and main mathematical difficulties in the proof. It should be required to develop global smooth approximate solutions through an intricate limiting process. Our strategy contains two fundamental components: an application of the local existence theory with strictly positive initial density and the utilization of a blow-up criterion for solutions (cf. Lemma \ref{l2.1}). This framework enables us to build approximate solutions before passing to the limit as the initial density's lower bound approaches zero. The cornerstone of our analysis lies in obtaining uniform {\it a priori} estimates for the bulk viscosity, that remain valid throughout the incompressible limit. It should be pointed out that the crucial techniques of proofs in \cite{Hoff05,D12,P14} cannot be adopted to the situation treated here, since their arguments rely heavily on the small initial energy.
Consequently, some new observations and ideas are needed to establish the desired {\it a priori} estimates under the assumption \eqref{lam}.

It was shown in \cite{SW11} that if $0<T^*<\infty$ is the maximal existence time of strong solutions to \eqref{a1}--\eqref{a4}, then
\begin{align*}
\limsup\limits_{T\nearrow T^*}\|\rho\|_{L^\infty(0,T;L^\infty)}=\infty,
\end{align*}
which implies that the key issue is to obtain the time-independent upper bound of the density. To this end, one of important steps is to derive time-independent upper bound of $\|\nabla\mathbf{u}\|_{L^2}^2$ and especially $(\mu+\lambda)\|\divv\mathbf{u}\|_{L^2}^2$ (see Lemma \ref{l3.2}). Inspired by \cite{P14}, we start with the basic energy estimate, in which we observe that the Lebesgue measure of $\|\nabla\mathbf{u}\|_{L^2}$ with respect to time should be small when $\|\nabla\mathbf{u}\|_{L^2}$ is suitably large. Based on this observation, we thus divide the arguments concerning the term $\int\rho \dot{\mathbf{u}}\cdot(\mathbf{u}\cdot\nabla)\mathbf{u}\mathrm{d}\mathbf{x}$ occurred in \eqref{z3.9} into two cases: $\|\nabla\mathbf{u}\|_{L^2}\leq1$ and $\|\nabla\mathbf{u}\|_{L^2}\geq1$, which can be controlled through $\frac{1}{(2\mu+\lambda)^2}\|P(\rho)-P(\tilde{\rho})\|_{L^4}^4$ (see  \eqref{z3.16}--\eqref{z3.18}) with the aid of a Desjardins-type logarithmic interpolation inequality (see Lemma $\ref{log}$). Inspired by \cite{HWZ}, this in turn motivates us to propose the \textit{a priori hypothesis} for the integrability in time of $\frac{1}{(2\mu+\lambda)^2}\|P(\rho)-P(\tilde{\rho})\|_{L^4}^4$ (see \eqref{z3.1}). To complete the proof of the \textit{a priori hypothesis} (that is, one needs to show \eqref{z3.2}), we observe from the definition of the \textit{effective viscous flux} in \eqref{z1.5}$_2$ that
\begin{equation*}
  \divv\mathbf{u}=\frac{-(-\Delta)^{-1}\divv(\rho\dot{\mathbf{u}})+P(\rho)-P(\tilde{\rho})}{2\mu+\lambda},
\end{equation*}
which further inspires us to establish the estimates on the material derivative of the velocity (see Lemmas $\ref{l3.2}$ and $\ref{l3.4}$).

However, since the bulk viscosity is often coupled with the divergence of the velocity while being inherently contained in the \textit{effective viscous flux}, it seems hard to extract information from $\divv \mathbf{u}$, and more fundamentally, from $\nabla\mathbf{u}$. Thus, some new difficulties will occur when one tackles the $L^\infty(0,\min\{1,T\};L^2)$-norm for the material derivative of the velocity (see Lemma $\ref{l3.4}$) in order to isolate $\lambda$ (see, e.g., \eqref{z3.35}).

First, it is very hard to obtain the estimates for the term
\begin{equation*}
  (\mu+\lambda)\sigma\int\dot{u}^j\left(\partial_j\divv\mathbf{u}_t +\divv(\mathbf{u}\partial_j\divv\mathbf{u})\right)\mathrm{d}\mathbf{x}.
\end{equation*}
To overcome this obstacle, motivated by \cite{HWZ}, we shall consider (see \eqref{z3.33})
\begin{equation*}
  (\mu+\lambda)\int(\divf\mathbf{u})^2\mathrm{d}\mathbf{x}~~~~
\text{rather than}
~~~~(\mu+\lambda)\int(\divv\dot{\mathbf{u}})^2\mathrm{d}\mathbf{x}.
\end{equation*}
But this makes the estimates more delicate and complicated (see \eqref{z3.34}--\eqref{z3.35}, or \eqref{z3.37}--\eqref{z3.39}).

Second, we have to use the Hodge-type decomposition to isolate the ``bad" terms (divergence-free part) from $\nabla\mathbf{u}$, which leads to the emergence of additional challenging terms. As stated in Remark \ref{r1.4}, the presence of the integral terms such as
\begin{equation*}
\int F(\mathcal{P}\mathbf{u})^i_j(\mathcal{P}\mathbf{u})^j_i\mathrm{d}\mathbf{x}
\end{equation*}
is analytically intractable because the divergence-free part
cannot restrain the bulk viscosity coefficient. To be concrete, only partial information (curl-free part) from $\|\nabla\mathbf{u}\|_{L^p}$ can be controlled in combination with the bulk viscosity.
To overcome it, we have to resort to boundary conditions for resolution. From the boundary condition \eqref{a3} and the flatness of the boundary, we observe that
\begin{align*}
\mathbf{u}\cdot\nabla(\mathcal{P}\mathbf{u})\cdot\mathbf{n}=(\mathcal{P}\mathbf{u})^{i}(\mathcal{P}\mathbf{u})_{i}^{j}n^j
  +(\mathcal{Q}\mathbf{u})^{i}(\mathcal{P}\mathbf{u})_{i}^{j}n^j=0 \ \ \text{on}\ \partial\mathbb{R}^2_+.
\end{align*}
Hence we can transfer the challenge to the treatment of the boundary terms (see \eqref{z3.43}--\eqref{z3.44})
\begin{align*}
\int_{\partial\mathbb{R}^2_+} F(\mathcal{P}\mathbf{u})\cdot\nabla(\mathcal{P}\mathbf{u})\cdot\mathbf{n}\mathrm{ds}
=-\int_{\partial\mathbb{R}^2_+} F(\mathcal{Q}\mathbf{u})\cdot\nabla(\mathcal{P}\mathbf{u})\cdot\mathbf{n}\mathrm{ds},
\end{align*}
which can be estimated by properly utilizing the curl-free part (see \eqref{z3.41} and \eqref{z3.43}). In particular, a crucial treatment (see \eqref{z3.39}) is required to generate the boundary term
\begin{align*}
\int_{\partial\mathbb{R}^2_+} F(\mathcal{P}\mathbf{u})\cdot\nabla(\mathcal{P}\mathbf{u})_t\cdot\mathbf{n}\mathrm{ds}~~~~
\text{rather than}
~~~~
\int_{\partial\mathbb{R}^2_+} F(\mathcal{P}\mathbf{u})_t\cdot\nabla(\mathcal{P}\mathbf{u})
\cdot\mathbf{n}\mathrm{ds},
\end{align*}
which manifests the higher regularity of weak solutions established here. The fundamental origin is that we can
extract the information concerning bulk viscosity from
$\|\nabla\mathcal{Q}\mathbf{u}\|_{L^p}$ rather than $\|\nabla\mathcal{Q}\mathbf{u}_t\|_{L^p}$ within this framework.
Then we succeed in deriving the desired estimates on $L^\infty(0,\min\{1,T\};L^2)$-norm (see Lemma $\ref{l3.4}$).

Having these time-independent estimates at hand, we can
obtain the time-independent upper bound of the density by applying Lagrangian
coordinates technique used in \cite{DE97} (see Lemma $\ref{l3.5}$). Consequently, the proof of \textit{a priori hypothesis} is complete once the bulk viscosity is properly large (see {\it Proof of Proposition \ref{p3.1}}). It should be emphasized that the effective viscous flux and Desjardins-type logarithmic interpolation inequality play essential roles in our analysis.

The rest of the paper is organized as follows. In the next section, we recall some known facts and elementary inequalities that will be used later. Section \ref{sec3} is devoted to establishing {\it a priori} estimates. The proofs of Theorems \ref{t1.1} and \ref{t1.2} are presented in Sections \ref{sec4} and \ref{sec5}, respectively.

\section{Preliminaries}\label{sec2}

In this section we collect some facts and elementary inequalities that will be used frequently later.

\subsection{Auxiliary results and inequalities}

In this subsection we review some known facts and inequalities. First of all, similar to the proofs in \cite{MN80,SW11}, we have the following results concerning the local existence and the possible breakdown of strong solutions to the problem \eqref{a1}--\eqref{a4}.
\begin{lemma}\label{l2.1}
Assume that
\begin{equation*}
(\rho_0-\tilde{\rho})\in H^{2},~~\inf\limits_{\mathbf{x}\in\mathbb{R}^2_+}\rho_0(\mathbf{x})>0,~~ \mathbf{u}_0\in H^2\cap \widetilde{H}^1,
\end{equation*}
then there exists a positive constant $T$ such that the problem \eqref{a1}--\eqref{a4} admits a unique strong solution $(\rho,\mathbf{u})$ satisfying
\begin{equation*}
(\rho-\tilde{\rho},\mathbf{u})\in C([0,T]; H^{2}),\ \ \inf_{\mathbb{R}^2_+\times[0,T]}\rho(\mathbf{x},t)\geq\frac{1}{2}
\inf_{\mathbf{x}\in\mathbb{R}^2_+}\rho_0(\mathbf{x})>0.
\end{equation*}
Moreover, if $T^*$ is the maximal time of existence, then it holds that
\begin{equation*}
\lim\sup_{T\nearrow T^*}\|\rho\|_{L^\infty(0,T;L^\infty)}=\infty.
\end{equation*}
\end{lemma}

The following well-known Gagliardo--Nirenberg inequality (see \cite{NI1959}) will be used.
\begin{lemma}\label{GN}
For $p\in [2, \infty)$, $q\in(1, \infty)$, and $r\in (2, \infty)$, there exists some generic constant $C>0$ which may depend on $p$, $q$, and $r$ such that, for $f\in H^1$ and $g\in L^q\cap D^{1, r}$,
\begin{gather*}
\|f\|_{L^p}\leq C\|f\|_{L^2}^\frac{2}{p}\|\nabla f\|_{L^2}^{1-\frac{2}{p}},\ \ \
\|g\|_{L^\infty}\leq C\|g\|_{L^q}^\frac{q(r-2)}{2r+q(r-2)}\|\nabla g\|_{L^r}^\frac{2r}{2r+q(r-2)}.
\end{gather*}
\end{lemma}

Next, we have the following Desjardins-type logarithmic interpolation inequality, which extends the case of two-dimensional torus $\mathbb{T}^2$ in \cite[Lemma 2]{DE97} (see also \cite[Lemma 1]{D1997}) to the half-plane.
\begin{lemma}\label{log}
Assume that $0\le\rho\le\hat{\rho}$ and $\mathbf{u}\in H^1(\mathbb{R}^2_+)$ satisfies the boundary condition \eqref{a3}, then it holds that
\begin{equation}\label{z2.1}
\|\sqrt\rho\mathbf{u}\|_{L^4}^2\leq C(\hat{\rho})(1+\|\sqrt\rho\mathbf{u}\|_{L^2})\|\mathbf{u}\|_{H^1}
\sqrt{\ln\big(2+\|\mathbf{u}\|_{H^1}^2\big)}.
\end{equation}
\end{lemma}
\begin{proof}
Let $\varphi(\mathbf{x})\in L^\infty$ be a nonnegative smooth function with
\begin{equation*}
\int_{\mathbb{R}^2_+}\varphi(\mathbf{x})\mathrm{d}\mathbf{x}=1.
\end{equation*}
For positive integer $n$, define $\varphi_n(\mathbf{x})= n^2\varphi(n\mathbf{x})$ and $\mathbf{u}_n(\mathbf{x}) = \mathbf{u} \ast \varphi_n(\mathbf{x})=\int_{\mathbb{R}^2_+} \mathbf{u}(\mathbf{x}-\mathbf{y})\varphi_n(\mathbf{y})\mathrm{d}\mathbf{y}$. Then it follows from H\"older's inequality and Lemma $\ref{GN}$ that
\begin{align}
\|\sqrt{\rho}\mathbf{u}\|_{L^4}^2&=\|\rho|\mathbf{u}|^2\|_{L^2}\notag\\
&\leq \|\rho\mathbf{u}\cdot(\mathbf{u}-\mathbf{u}_n)\|_{L^2}
+\|\rho\mathbf{u}\cdot\mathbf{u}_n\|_{L^2}\notag\\
&\leq \|\rho\mathbf{u}\|_{L^4}\|\mathbf{u}-\mathbf{u}_n\|_{L^4}+\|\rho\mathbf{u}\|_{L^2}
\|\mathbf{u}_n\|_{L^\infty}\notag\\
&\leq C\hat{\rho}^\frac{1}{2}\|\mathbf{u}-\mathbf{u}_n\|_{L^2}^{\frac{1}{2}}
\|\nabla(\mathbf{u}-\mathbf{u}_n)\|_{L^2}^{\frac{1}{2}}\|\sqrt{\rho}\mathbf{u}\|_{L^4}+
\hat{\rho}^\frac{1}{2}\|\sqrt{\rho}\mathbf{u}\|_{L^2}
\|\mathbf{u}_n\|_{L^\infty}\notag\\
&\leq C(\hat{\rho})\bigg(\frac{\|\nabla\mathbf{u}\|_{L^2}}{n}\bigg)^\frac12\|\nabla\mathbf{u}\|_{L^2}^{\frac12}\|\sqrt{\rho}\mathbf{u}\|_{L^4}+
\hat{\rho}^\frac{1}{2}\|\sqrt{\rho}\mathbf{u}\|_{L^2}
\|\mathbf{u}_n\|_{L^\infty}\notag\\
&\leq \frac{1}{2}\|\sqrt{\rho}\mathbf{u}\|_{L^4}^2+\frac{C(\hat{\rho})}{n}\|\mathbf{u}\|_{H^1}^2+
\hat{\rho}^\frac{1}{2}\|\sqrt{\rho}\mathbf{u}\|_{L^2}
\|\mathbf{u}_n\|_{L^\infty}\notag,
\end{align}
where in the forth inequality we have used
\begin{equation*}
\|\mathbf{u}-\mathbf{u}_{n}\|_{L^{2}}=\|\mathbf{u}-\mathbf{u}*\varphi_{n}\|_{L^{2}}\leq \frac{C}{n}\|\nabla\mathbf{u}\|_{L^{2}},
\end{equation*}
which can be deduced from Lemma 1.50 and Theorem 1.52 in \cite{M97}. Hence,
\begin{equation}\label{z2.2}
\|\sqrt{\rho}\mathbf{u}\|_{L^4}^2\leq \frac{C(\hat{\rho})}{n}\|\mathbf{u}\|_{H^1}^2+
2\hat{\rho}^\frac{1}{2}\|\sqrt{\rho}\mathbf{u}\|_{L^2}
\|\mathbf{u}_n\|_{L^\infty}.
\end{equation}

Now we consider the following Moser--Trudinger inequality in \cite{TR67}
\begin{equation*}
\int_{\mathbb{R}^2}\bigg(\exp\bigg(\frac{c|f(\mathbf{x})|^2}{\|f\|_{H^1
(\mathbb{R}^2)}^2}\bigg)-1\bigg)\mathrm{d}\mathbf{x}\leq C, ~\text{for some constants}~c, C>0.
\end{equation*}
For a function $\mathbf{u}=(u^1,u^2)$ satisfying the boundary condition \eqref{a3}, we extend it to be in $\mathbb{R}^2$ as follows
\begin{align*}
\begin{array}{ll}
\tilde{u}^1(x_1,x_2)=
\begin{cases}
u^1(x_1,x_2),& x_2\geq0,\\
u^1(x_1,-x_2)e^{2x_2},& x_2<0,
\end{cases}\ \ \ \ \
&\tilde{u}^2(x_1,x_2)=
\begin{cases}
u^2(x_1,x_2),& x_2\geq0,\\
-u^2(x_1,-x_2),& x_2<0,
\end{cases}
\end{array}
\end{align*}
where $\tilde{\mathbf{u}}=(\tilde{u}^1,\tilde{u}^2)\in H^{1}(\mathbb{R}^2)$ satisfies $(\partial_2 \tilde{u}^1,\tilde{u}^2)=(u^1,0)$ on $\partial\mathbb{R}^2_+$. Obviously, we have
$$
\|\mathbf{u}\|_{H^{1}}\leq \|\tilde{\mathbf{u}}\|_{H^{1}(\mathbb{R}^2)}\leq C\|\mathbf{u}\|_{H^1},
$$
which implies that
\begin{equation*}
\int_{\mathbb{R}^2_+}\bigg(\exp\bigg(\frac{c|\mathbf{u}(\mathbf{x})|^2}
{\|\mathbf{u}\|_{H^1}^2}\bigg)-1\bigg)\mathrm{d}\mathbf{x}\leq C, ~~\text{for some constants}~c, C>0.
\end{equation*}
Therefore, in view of Young's inequality
\begin{equation*}
ab\leq\exp(a^2)-1+Cb\sqrt{\ln(2+b)}, ~~\text{for some constant}~C>0,
\end{equation*}
one gets that
\begin{align}
|\mathbf{u}_n(\mathbf{x})|&=\bigg|\int_{\mathbb{R}^2_+}\mathbf{u}
(\mathbf{x}-\mathbf{\mathbf{y}})\varphi_n(\mathbf{\mathbf{y}})\mathrm{d}\mathbf{y}
\bigg|\notag\\
&\leq\int_{\mathbb{R}^2_+}\frac{\sqrt{c}|\mathbf{u}(\mathbf{x}-\mathbf{y})|}{\|\mathbf{u}\|_{H^1}}\frac{\|\mathbf{u}\|_{H^1}}{\sqrt{c}}\varphi_n(\mathbf{y})\mathrm{d}\mathbf{y}
\notag\\
&\leq\int_{\mathbb{R}^2_+}\left(\exp\left(\frac{c|\mathbf{u}(\mathbf{x}-\mathbf{y})|^2}{\|\mathbf{u}\|_{H^1}^2}\right)-1\right)\mathrm{d}\mathbf{y}
+C\int_{\mathbb{R}^2_+}\|\mathbf{u}\|_{H^1}\varphi_n(\mathbf{y})\sqrt{\ln(2+\|\mathbf{u}\|_{H^1}\varphi_n(\mathbf{y}))}\mathrm{d}\mathbf{y}\notag\\
&\leq C+C\|\mathbf{u}\|_{H^1}\sqrt{\ln(2+n^2\|\mathbf{u}\|_{H^1})}\int_{\mathbb{R}^2_+}\varphi_n(\mathbf{y})\mathrm{d}\mathbf{y}\notag\\
&=C+C\|\mathbf{u}\|_{H^1}\sqrt{\ln(2+n^2\|\mathbf{u}\|_{H^1})}\notag,
\end{align}
which combined with \eqref{z2.2} yields that
\begin{align*}
\|\sqrt{\rho}\mathbf{u}\|_{L^4}^2&\leq \frac{C}{n}\|\mathbf{u}\|_{H^1}^2+C\|\sqrt{\rho}\mathbf{u}\|_{L^2}
\big(1+\|\mathbf{u}\|_{H^1}
\sqrt{\ln(2+n^2\|\mathbf{u}\|_{H^1})}\big)\notag\\
&\leq C\left(1+\|\sqrt{\rho}\mathbf{u}\|_{L^2}\right)\|\mathbf{u}\|_{H^1}
\left(\frac{\|\mathbf{u}\|_{H^1}}{n}+\sqrt{\ln(2+n^2\|\mathbf{u}\|_{H^1})}\right).
\end{align*}
Thus, we obtain the desired \eqref{z2.1} by choosing the positive integer $n=\left[\left(2+\|\mathbf{u}\|_{H^1}^2\right)^\frac{1}{2}\right]$.
\end{proof}

It should be emphasized that the statement of Lemma \ref{log} has nothing to do with the system \eqref{a1}. It is simply an interpolation between different functional spaces.

Next, the following Hodge-type decomposition is given in \cite{WW92}.
\begin{lemma}\label{Hodge}
Let $1<q<\infty$. For $\mathbf{u}\in W^{1,q}(\mathbb{R}^2_+)$ with $(\mathbf{u}\cdot\mathbf{n})|_{\partial\mathbb{R}^2_+}=0$, there exists a positive constant $C=C(q)$ such that
\begin{equation*}
\|\nabla \mathbf{u}\|_{L^q}\leq
C\big(\|\divv \mathbf{u}\|_{L^q}+\|\curl \mathbf{u}\|_{L^q}\big).
\end{equation*}
\end{lemma}

Finally, we recall the following commutator estimates in \cite[Lemma 4.3]{F04}, which play an important role in the mollifier arguments.
\begin{lemma}\label{lcom}
Let $\Omega\subset \mathbb{R}^2$ be a domain.
Let $\rho \in L^{p}(\Omega)$ and ${\bf u} \in W^{1,q}(\Omega)$ be given functions such that $1\le p,q<\infty$ and $\frac{1}{p}+\frac{1}{q}\le 1$. For any $\epsilon>0$, then we have
\begin{equation*}
\|\divv[\rho{\bf u}]_\epsilon-\divv\left([\rho]_\epsilon{\bf u}\right)\|_{L^1(K)}\le C(K)\|\rho\|_{L^p(\Omega)}\|{\bf u}\|_{W^{1,q}(\Omega)},
\end{equation*}
and
\begin{equation*}
\divv[\rho{\bf u}]_\epsilon-\divv\left([\rho]_\epsilon{\bf u}\right)\rightarrow 0\ \text{ in}\ \ L^1(K)\ \ \text{as}\ \epsilon\rightarrow0
\end{equation*}
for any compact set $K\subset \Omega$.
\end{lemma}

\subsection{Uniform estimates for $F$, $\curl\mathbf{u}$, and $\nabla\mathbf{u}$}
For the effective viscous flux $F$, $\mathbf{\curl u}$, and $\mathbf{\nabla u}$, we have the following estimates.
\begin{lemma}\label{E0}
Let $(\rho,\mathbf{u})$ be a smooth solution to the problem \eqref{a1}--\eqref{a4}. Then, for any $2\leq p<\infty$, there exists a generic positive constant $C$ depending only on $p$ and $\mu$ such that
\begin{equation}\label{E1}
\|\nabla F\|_{L^p}+\|\nabla\curl \mathbf{u}\|_{L^p}+\|\nabla^2\mathcal{P}\mathbf{u}\|_{L^p}\leq C\left(\|\rho\dot{\mathbf{u}}\|_{L^p}+\|\nabla \mathbf{u}\|_{L^p}\right),
\end{equation}
\begin{equation}\label{E2}
\|\curl\mathbf{u}\|_{L^p}+\|\nabla\mathcal{P}\mathbf{u}\|_{L^p}\leq C\Big(\|\rho\dot{\mathbf{u}}\|_{L^2}^{1-\frac{2}{p}}\|\nabla \mathbf{u}\|_{L^2}^\frac{2}{p}+\|\nabla \mathbf{u}\|_{L^2}\Big),
\end{equation}
\begin{equation}\label{E3}
\|F\|_{L^p}\leq
C(2\mu+\lambda)^\frac{2}{p}\Big(\|\rho\dot{\mathbf{u}}\|_{L^2}^{1-\frac{2}{p}}
\|\nabla\mathbf{u}\|_{L^2}^\frac{2}{p}+\|\nabla\mathbf{u}\|_{L^2}\Big)
+C\|P-P(\tilde{\rho})\|_{L^2}^\frac{2}{p}
\Big(\|\rho\dot{\mathbf{u}}\|_{L^2}^{1-\frac{2}{p}}
+\|\nabla\mathbf{u}\|_{L^2}^{1-\frac{2}{p}}\Big),
\end{equation}
\begin{equation}\label{E4}
\begin{aligned}[b]
\|\nabla \mathbf{u}\|_{L^p}
&\leq C\|\rho\dot{\mathbf{u}}\|_{L^2}^{1-\frac{2}{p}}
\|\nabla\mathbf{u}\|_{L^2}^\frac{2}{p}+C\|\nabla \mathbf{u}\|_{L^2}
+\frac{C}{2\mu+\lambda}\|P-P(\tilde{\rho})\|_{L^p}
\\&\quad
+\frac{C}{2\mu+\lambda}\|\rho\dot{\mathbf{u}}\|_{L^2}^{1-\frac{2}{p}}
\|P-P(\tilde{\rho})\|_{L^2}^\frac{2}{p}
+\frac{C}{2\mu+\lambda}\|\nabla\mathbf{u}\|_{L^2}^{1-\frac{2}{p}}
\|P-P(\tilde{\rho})\|_{L^2}^\frac{2}{p}.
\end{aligned}
\end{equation}
\end{lemma}
\begin{proof}
By $\eqref{a1}_2$ and the boundary condition \eqref{a3}, one has that
\begin{align*}
\begin{cases}
\mu\Delta(\omega+u^1)=\curl(\rho{\dot{\mathbf{u}}})+\mu\Delta u^1, &\mathbf{x}\in \mathbb{R}^2_+,\\
\omega+u^1=0, & x_2=0.
\end{cases}
\end{align*}
Then the standard $L^p$-theory gives that
\begin{equation*}
  \|\nabla\omega\|_{L^p}\leq C\left(\|\rho\dot{\mathbf{u}}\|_{L^p}+\|\nabla \mathbf{u}\|_{L^p}\right),
\end{equation*}
which leads to
\begin{equation*}
  \|\nabla F\|_{L^p}\leq \|\rho\dot{\mathbf{u}}\|_{L^p}+\mu\|\nabla^\bot \omega\|_{L^p}
  \leq C\left(\|\rho\dot{\mathbf{u}}\|_{L^p}+\|\nabla \mathbf{u}\|_{L^p}\right),
\end{equation*}
due to $\nabla F=\rho\dot{\mathbf{u}}+\mu\nabla^\bot \omega$.

Next, we derive from Lemma \ref{GN}, \eqref{E1}, and Young's inequality that
\begin{align*}
\|\curl \mathbf{u}&\|_{L^p}\leq C\|\curl \mathbf{u}\|_{L^2}^\frac{2}{p}\|\nabla\curl \mathbf{u}\|_{L^2}^{1-\frac{2}{p}}
\leq C\Big(\|\rho\dot{\mathbf{u}}\|_{L^2}^{1-\frac{2}{p}}\|\nabla \mathbf{u}\|_{L^2}^\frac{2}{p}+\|\nabla \mathbf{u}\|_{L^2}\Big),\\
\|F\|_{L^p}&\leq C\|F\|_{L^2}^\frac{2}{p}\|\nabla F\|_{L^2}^{1-\frac{2}{p}}
\leq C\left((2\mu+\lambda)\|\divv\mathbf{u}\|_{L^2}+\|P-P(\tilde{\rho})\|_{L^2}\right)^\frac{2}{p}
\left(\|\rho\dot{\mathbf{u}}\|_{L^2}+\|\nabla \mathbf{u}\|_{L^2}\right)^{1-\frac{2}{p}}\\
&\leq C(2\mu+\lambda)^\frac{2}{p}\Big(\|\rho\dot{\mathbf{u}}\|_{L^2}^{1-\frac{2}{p}}
\|\nabla\mathbf{u}\|_{L^2}^\frac{2}{p}+\|\nabla\mathbf{u}\|_{L^2}\Big)
+C\|P-P(\tilde{\rho})\|_{L^2}^\frac{2}{p}
\Big(\|\rho\dot{\mathbf{u}}\|_{L^2}^{1-\frac{2}{p}}
+\|\nabla\mathbf{u}\|_{L^2}^{1-\frac{2}{p}}\Big),
\end{align*}
as the desired \eqref{E2} and \eqref{E3}. Combining these with Lemma $\ref{Hodge}$ indicates that
\begin{align*}
\|\nabla\mathbf{u}\|_{L^p}&\leq C(\|\divv\mathbf{u}\|_{L^p}+\|\curl\mathbf{u}\|_{L^p})\notag\\
&\leq \frac{C}{2\mu+\lambda}(\|F\|_{L^p}+\|P-P(\tilde{\rho})\|_{L^p})+C\|\curl\mathbf{u}\|_{L^p}\\
&\leq
\frac{C}{(2\mu+\lambda)^{1-\frac{2}{p}}}\Big(\|\rho\dot{\mathbf{u}}\|_{L^2}^{1-\frac{2}{p}}
\|\nabla\mathbf{u}\|_{L^2}^\frac{2}{p}+\|\nabla\mathbf{u}\|_{L^2}\Big)+\frac{C}{2\mu+\lambda}\Big(\|\rho\dot{\mathbf{u}}\|_{L^2}^{1-\frac{2}{p}}\|P-P(\tilde{\rho})\|_{L^2}^\frac{2}{p}
\notag\\&\quad+\|\nabla\mathbf{u}\|_{L^2}^{1-\frac{2}{p}}\|P-P(\tilde{\rho})\|_{L^2}^\frac{2}{p}
+\|P-P(\tilde{\rho})\|_{L^p}\Big)+C\Big(\|\rho\dot{\mathbf{u}}\|_{L^2}^{1-\frac{2}{p}}\|\nabla \mathbf{u}\|_{L^2}^\frac{2}{p}+\|\nabla \mathbf{u}\|_{L^2}\Big)\\
&\leq C\|\rho\dot{\mathbf{u}}\|_{L^2}^{1-\frac{2}{p}}
\|\nabla\mathbf{u}\|_{L^2}^\frac{2}{p}+C\|\nabla \mathbf{u}\|_{L^2}
+\frac{C}{2\mu+\lambda}\|P-P(\tilde{\rho})\|_{L^p}
\notag\\&\quad
+\frac{C}{2\mu+\lambda}\|\rho\dot{\mathbf{u}}\|_{L^2}^{1-\frac{2}{p}}
\|P-P(\tilde{\rho})\|_{L^2}^\frac{2}{p}
+\frac{C}{2\mu+\lambda}\|\nabla\mathbf{u}\|_{L^2}^{1-\frac{2}{p}}
\|P-P(\tilde{\rho})\|_{L^2}^\frac{2}{p},
\end{align*}
where we have used the fact $2\mu+\lambda\geq\mu>0$.
\end{proof}

\section{\textit{A priori} estimates}\label{sec3}
In this section we will show some necessary {\it a priori} bounds for the strong solutions guaranteed by Lemma $\ref{l2.1}$ to the problem \eqref{a1}--\eqref{a4}. Let us point out that these bounds are independent of the bulk viscosity $\lambda$, the lower bound of $\rho$, the initial regularity, and the time of existence. Furthermore, let $T>0$ be fixed and $(\rho, \mathbf{u})$ be the strong solution to \eqref{a1}--\eqref{a4} in $\mathbb{R}^2_+\times(0, T]$, we can obtain the following key {\it a priori} estimates on $(\rho, \mathbf{u})$.
\begin{proposition}\label{p3.1}
Under the conditions of Theorem $\ref{t1.1}$, if $(\rho, \mathbf{u})$ is a strong solution to the initial-boundary value problem \eqref{a1}--\eqref{a4} satisfying
\begin{align}\label{z3.1}
\sup_{\mathbb{R}^2_+\times[0,T]}\rho\le2\hat{\rho},\ \ \frac{1}{(2\mu+\lambda)^2}\int_{0}^{T}
\|P-P(\tilde{\rho})\|_{L^4}^4\mathrm{d}t\le2,
\end{align}
then one has that
\begin{align}\label{z3.2}
\sup_{\mathbb{R}^2_+\times[0,T]}\rho\le\frac{7}{4}\hat{\rho},\ \ \frac{1}{(2\mu+\lambda)^2}\int_{0}^{T}
\|P-P(\tilde{\rho})\|_{L^4}^4\mathrm{d}t\le1.
\end{align}
\end{proposition}

Before proving Proposition \ref{p3.1}, we establish some necessary \textit{a priori} estimates, see Lemmas \ref{l3.1}--\ref{l3.4} below. Let us start with the elementary energy estimate of $(\rho, \mathbf{u})$.
\begin{lemma}\label{l3.1}
It holds that
\begin{align}\label{z3.3}
\sup_{0\le t\le T}\int\bigg(\frac{1}{2}\rho |\mathbf{u}|^2+G(\rho)\bigg)\mathrm{d}\mathbf{x}+\int_0^T\left[\mu\|\nabla \mathbf{u}\|_{L^2}^2+(\mu+\lambda)\|\divv \mathbf{u}\|_{L^2}^2\right]\mathrm{d}t
+\mu\int_0^T\int_{\partial\mathbb{R}^2_+}|\mathbf{u}|^2\mathrm{ds}\mathrm{d}t
\leq C_0.
\end{align}
\end{lemma}
\begin{proof}
Multiplying $\eqref{a1}_2$ by $\mathbf{u}$ and integrating the resultant over $\mathbb{R}^2_+$, we deduce from \eqref{a3} and \eqref{a4} that
\begin{align}\label{z3.4}
\frac{1}{2}\frac{\mathrm{d}}{\mathrm{d}t}\int\rho|\mathbf{u}|^2\mathrm{d}\mathbf{x}
+\int\mathbf{u}\cdot\nabla P\mathrm{d}\mathbf{x}+\int\left[\mu|\nabla\mathbf{u}|^2
+(\mu+\lambda)(\divv\mathbf{u})^2\right]\mathrm{d}\mathbf{x}
+\mu\int_{\partial\mathbb{R}^2_+}|\mathbf{u}|^2\mathrm{ds}=0.
\end{align}
By $\eqref{a1}_1$ and the definition of $G(\rho)$ in \eqref{z1.6}, one has that
\begin{equation}\label{z3.5}
(G(\rho))_t+\divv (G(\rho)\mathbf{u})+(P-P(\tilde{\rho}))\divv \mathbf{u}=0.
\end{equation}
Integrating \eqref{z3.5} over $\mathbb{R}^2_+$ and then adding \eqref{z3.4}, we get that
\begin{align}\label{z3.6}
\frac{\mathrm{d}}{\mathrm{d}t}\int\bigg(\frac{1}{2}\rho |\mathbf{u}|^2+G(\rho)\bigg)\mathrm{d}\mathbf{x}+\int\left[\mu|\nabla\mathbf{u}|^2
+(\mu+\lambda)(\divv\mathbf{u})^2\right]\mathrm{d}\mathbf{x}
+\mu\int_{\partial\mathbb{R}^2_+}|\mathbf{u}|^2\mathrm{ds}=0,
\end{align}
which yields \eqref{z3.3} after integrating \eqref{z3.6} with respect to $t$ over $(0,T)$.
\end{proof}

The following result concerns the time-independent estimate on $\nabla \mathbf{u}$ in $L^\infty(0, T;L^2)$.
\begin{lemma}\label{l3.2}
Let \eqref{z3.1} be satisfied, then there exists a positive constant $D_2$ depending only on $\tilde{\rho}$, $\hat{\rho}$, $a$, $\gamma$, and $\mu$ such that
\begin{align}\label{z3.7}
\sup_{0\le t\le T}\left[\mu\|\nabla \mathbf{u}\|_{L^2}^2+(\mu+\lambda)\|\divv \mathbf{u}\|_{L^2}^2\right]+\int_0^T\|\sqrt{\rho} \dot{\mathbf{u}}\|_{L^2}^2\mathrm{d}t\le (2+M)^{e^{2D_2(1+C_0)^5}}
\end{align}
provided that $\lambda$ satisfies \eqref{lam} with $D\geq D_2$.
\end{lemma}
\begin{proof}
According to $\eqref{a1}_1$, one can rewrite $\eqref{a1}_2$ as
\begin{equation}\label{z3.8}
\rho\mathbf{u}_t+\rho\mathbf{u}\cdot\nabla\mathbf{u}-\mu\Delta\mathbf{u}-(\mu+\lambda)\nabla\divv \mathbf{u}+\nabla P=0.
\end{equation}
Multiplying \eqref{z3.8} by $\mathbf{u}_t$ and integrating (by parts) the resultant over $\mathbb{R}^2_+$, we obtain that
\begin{align}\label{z3.9}
&\frac{1}{2}\frac{\mathrm{d}}{\mathrm{d}t}\bigg(\int\left[\mu|\nabla \mathbf{u}|^2+(\mu+\lambda)(\divv \mathbf{u})^2\right]\mathrm{d}\mathbf{x}
+\mu\int_{\partial\mathbb{R}^2_+}|\mathbf{u}|^2\mathrm{ds}\bigg)+\int\rho |\dot{\mathbf{u}}|^2\mathrm{d}\mathbf{x}\notag\\
&=-\int\mathbf{u}_t\cdot\nabla P\mathrm{d}\mathbf{x}+\int\rho \dot{\mathbf{u}}\cdot(\mathbf{u}\cdot\nabla)\mathbf{u}\mathrm{d}\mathbf{x}.
\end{align}

For the first term on the right-hand side of \eqref{z3.9}, it follows from \eqref{a3}--\eqref{z1.5} that
\begin{align}\label{z3.10}
&-\int\mathbf{u}_t\cdot\nabla P\mathrm{d}\mathbf{x}=\int (P-P(\tilde{\rho}))\divv\mathbf{u}_t\mathrm{d}\mathbf{x}\notag\\
&=\frac{\mathrm{d}}{\mathrm{d}t}\int(P-P(\tilde{\rho}))\divv\mathbf{u} \mathrm{d}\mathbf{x}-\int\divv\mathbf{u}P'(\rho)\rho_t\mathrm{d}\mathbf{x}\notag\\
&=\frac{\mathrm{d}}{\mathrm{d}t}\int(P-P(\tilde{\rho}))\divv\mathbf{u} \mathrm{d}\mathbf{x}+\int\divv\mathbf{u}P'(\rho)\rho\divv\mathbf{u}
\mathrm{d}\mathbf{x}+\int\divv\mathbf{u}P'(\rho)\mathbf{u}\cdot\nabla\rho\mathrm{d}\mathbf{x}\notag\\
&=\frac{\mathrm{d}}{\mathrm{d}t}\int(P-P(\tilde{\rho}))\divv\mathbf{u}\mathrm{d}\mathbf{x}+
\int(\divv\mathbf{u})^2P'(\rho)\rho\mathrm{d}\mathbf{x}+
\int\mathbf{u}\cdot\nabla(P-P(\tilde{\rho}))\divv\mathbf{u}\mathrm{d}\mathbf{x}\notag\\
&=\frac{\mathrm{d}}{\mathrm{d}t}\int(P-P(\tilde{\rho}))\divv\mathbf{u}\mathrm{d}\mathbf{x}+
\int(\divv\mathbf{u})^2(P'(\rho)\rho-P+P(\tilde{\rho}))\mathrm{d}\mathbf{x}-
\int(P-P(\tilde{\rho}))\mathbf{u}\cdot\nabla\divv\mathbf{u}\mathrm{d}\mathbf{x}\notag\\
&=\frac{\mathrm{d}}{\mathrm{d}t}\int(P-P(\tilde{\rho}))\divv\mathbf{u}\mathrm{d}\mathbf{x}+
\int(\divv\mathbf{u})^2(P'(\rho)\rho-P+P(\tilde{\rho}))\mathrm{d}\mathbf{x}\notag\\
&\quad-\frac{1}{2\mu+\lambda}
\int (P-P(\tilde{\rho}))\mathbf{u}\cdot\nabla(F+P-P(\tilde{\rho}))\mathrm{d}\mathbf{x}\notag\\
&=\frac{\mathrm{d}}{\mathrm{d}t}\int(P-P(\tilde{\rho}))\divv\mathbf{u}\mathrm{d}\mathbf{x}+
\int(\divv\mathbf{u})^2(P'(\rho)\rho-P+P(\tilde{\rho}))\mathrm{d}\mathbf{x}-\frac{1}{2\mu+\lambda}
\int (P-P(\tilde{\rho}))\mathbf{u}\cdot\nabla F\mathrm{d}\mathbf{x}\notag\\
&\quad+\frac{1}{4\mu+2\lambda}
\int(P-P(\tilde{\rho}))^2\divv\mathbf{u}\mathrm{d}\mathbf{x}\notag\\
&\triangleq\frac{\mathrm{d}}{\mathrm{d}t}\int(P-P(\tilde{\rho}))\divv\mathbf{u} \mathrm{d}\mathbf{x}+\sum_{i=1}^{3}I_i.
\end{align}
By \eqref{z3.1}, we deduce that
\begin{gather}\label{z3.11}
I_1\leq C(\tilde{\rho},\hat{\rho})\|\divv \mathbf{u}\|_{L^2}^2,\\
I_3\leq\frac{C}{(2\mu+\lambda)^2}\|P-P(\tilde{\rho})\|_{L^4}^4+C\|\divv\mathbf{u}\|_{L^{2}}^{2}.\label{z3.12}
\end{gather}
For the treatment of $I_2$, taking advantage of \eqref{z1.10}, \eqref{z3.3}, and Lemma $\ref{GN}$, we find that
\begin{align}\label{z3.13}
\tilde{\rho}\int|\mathbf{u}|^2\mathrm{d}\mathbf{x}
&=\int\rho|\mathbf{u}|^2\mathrm{d}\mathbf{x}
+\int(\tilde{\rho}-\rho)|\mathbf{u}|^2\mathrm{d}\mathbf{x}\notag\\
&\leq\|\sqrt{\rho}\mathbf{u}\|_{L^2}^2
+\|\rho-\tilde{\rho}\|_{L^2}\|\mathbf{u}\|_{L^4}^2\notag\\
&\leq\|\sqrt{\rho}\mathbf{u}\|_{L^2}^2
+C\|\rho-\tilde{\rho}\|_{L^2}\|\mathbf{u}\|_{L^2}\|\nabla\mathbf{u}\|_{L^2}\notag
\\&\leq2C_0
+CC_0^\frac{1}{2}\|\mathbf{u}\|_{L^2}\|\nabla\mathbf{u}\|_{L^2},
\end{align}
which combined with the fact $\tilde{\rho}>0$ and Cauchy--Schwarz inequality ensures that
\begin{align*}
\|\mathbf{u}\|_{L^2}^2 \leq C(\tilde{\rho},\hat{\rho})C_0
\big(1+\|\nabla\mathbf{u}\|_{L^2}^2\big),
\end{align*}
and furthermore,
\begin{align}\label{z3.14}
\|\mathbf{u}\|_{H^1}\leq C(\tilde{\rho},\hat{\rho})(1+C_0)^\frac12
\big(1+\|\nabla\mathbf{u}\|_{L^2}\big).
\end{align}
As a result, we can infer from Lemmas $\ref{GN}$, $\ref{E0}$, and H\"older's inequality that
\begin{align}\label{z3.15}
I_{2}&\leq\frac{1}{2\mu+\lambda}\|P-P(\tilde{\rho})\|_{L^4}\|\mathbf{u}\|_{L^4}\|\nabla F\|_{L^2}\notag\\
&\leq\frac{C}{2\mu+\lambda}\|P-P(\tilde{\rho})\|_{L^4}\|\mathbf{u}\|_{L^2}^{\frac{1}{2}}
\|\nabla\mathbf{u}\|_{L^2}^{\frac{1}{2}}
\big(\|\sqrt{\rho}\dot{\mathbf{u}}\|_{L^2}+\|\nabla\mathbf{u}\|_{L^2}\big)\notag\\
&\leq\frac{C(1+C_0)^\frac14}{2\mu+\lambda}\|P-P(\tilde{\rho})\|_{L^4}
\big(1+\|\nabla\mathbf{u}\|_{L^2}\big)^\frac12
\|\nabla\mathbf{u}\|_{L^2}^{\frac{1}{2}}
\big(\|\sqrt{\rho}\dot{\mathbf{u}}\|_{L^2}+\|\nabla\mathbf{u}\|_{L^2}\big)\notag\\
&\leq\frac{1}{4}\|\sqrt{\rho}\dot{\mathbf{u}}\|_{L^2}^2
+C(1+C_0)\|\nabla\mathbf{u}\|_{L^{2}}^{2}\big(1+\|\nabla\mathbf{u}\|_{L^{2}}^{2}\big)
+\frac{C}{(2\mu+\lambda)^4}\|P-P(\tilde{\rho})\|_{L^4}^4.
\end{align}
Therefore, substituting \eqref{z3.11}, \eqref{z3.12}, and \eqref{z3.15} into \eqref{z3.10} yields that
\begin{align}\label{z3.16}
-\int\mathbf{u}_t\cdot\nabla P\mathrm{d}\mathbf{x}&\leq\frac{\mathrm{d}}{\mathrm{d}t}\int (P-P(\tilde{\rho}))\divv\mathbf{u} \mathrm{d}\mathbf{x}+
\frac{C}{(2\mu+\lambda)^2}\|P-P(\tilde{\rho})\|_{L^4}^4
\notag \\ & \quad+\frac{1}{4}\|\sqrt{\rho}\dot{\mathbf{u}}\|_{L^2}^2+
C(1+C_0)\|\nabla\mathbf{u}\|_{L^{2}}^{2}\big(1+\|\nabla\mathbf{u}\|_{L^{2}}^{2}\big).
\end{align}

Next, we deal with the second term on the right-hand side of \eqref{z3.9} in two cases.

\textbf{Case 1:} If $\|\nabla\mathbf{u}\|_{L^{2}}\leq1$, \eqref{z3.14} simplifies to
\begin{equation*}
\|\mathbf{u}\|_{H^1}\leq C(\tilde{\rho},\hat{\rho})(1+C_0)^\frac12,
\end{equation*}
then it follows from H\"older's inequality, Cauchy--Schwarz inequality, Lemma $\ref{GN}$, and \eqref{E4} that
\begin{align}\label{z3.17}
&\int\rho \dot{\mathbf{u}}\cdot(\mathbf{u}\cdot\nabla)\mathbf{u}\mathrm{d}\mathbf{x}\notag\\
&\leq C\|\sqrt{\rho}\dot{\mathbf{u}}\|_{L^2}\|\mathbf{u}\|_{L^{4}}\|\nabla\mathbf{u}\|_{L^{4}}\notag\\
&\leq C\|\sqrt{\rho}\dot{\mathbf{u}}\|_{L^2}\|\mathbf{u}\|_{L^{2}}^{\frac{1}{2}}
\|\nabla\mathbf{u}\|_{L^{2}}^{\frac{1}{2}}\Big(\|\sqrt{\rho}\dot{\mathbf{u}}\|_{L^2}^{\frac{1}{2}}
\|\nabla\mathbf{u}\|_{L^2}^\frac{1}{2}+\|\nabla\mathbf{u}\|_{L^2}\notag\\
&\quad+\frac{1}{2\mu+\lambda}\|\sqrt{\rho}\dot{\mathbf{u}}\|_{L^2}^{\frac{1}{2}}
\|P-P(\tilde{\rho})\|_{L^2}^\frac{1}{2}
+\frac{1}{2\mu+\lambda}\|\nabla\mathbf{u}\|_{L^2}^{\frac{1}{2}}
\|P-P(\tilde{\rho})\|_{L^2}^\frac{1}{2}
+\frac{1}{2\mu+\lambda}
\|P-P(\tilde{\rho})\|_{L^4}\Big)\notag\\
&\leq \frac{1}{4}\|\sqrt{\rho}\dot{\mathbf{u}}\|_{L^2}^2+C(1+C_0)\|\nabla \mathbf{u}\|_{L^2}^2+\frac{C}{(2\mu+\lambda)^4}\|P-P(\tilde{\rho})\|_{L^4}^4,
\end{align}
where in the last inequality we have used
\begin{equation*}
\|P-P(\tilde{\rho})\|_{L^2}\leq C(\tilde{\rho},\hat{\rho},a,\gamma)\|\rho-\tilde{\rho}\|_{L^2}\leq CC_0^{\frac12},
\end{equation*}
due to \eqref{z1.10} and \eqref{z3.3}.

\textbf{Case 2:} If $\|\nabla\mathbf{u}\|_{L^{2}}\ge1$, noting that
\begin{equation*}
\ln(2+bc)\le b\ln(2+c),~\mathrm{for}~b,c\ge1,
\end{equation*}
then one can deduce from Lemmas $\ref{log}$, $\ref{E0}$, $\ref{l3.1}$, and \eqref{z3.14} that
\begin{align}\label{z3.18}
&\int\rho \dot{\mathbf{u}}\cdot(\mathbf{u}\cdot\nabla)\mathbf{u}\mathrm{d}\mathbf{x}\notag\\
&\leq C\|\sqrt{\rho}\dot{\mathbf{u}}\|_{L^2}\|\sqrt{\rho}\mathbf{u}\|_{L^{4}}\|\nabla\mathbf{u}\|_{L^{4}}\notag\\
&\leq C\|\sqrt{\rho}\dot{\mathbf{u}}\|_{L^2}(1+\|\sqrt\rho\mathbf{u}\|_{L^2})^{\frac{1}{2}}
\|\mathbf{u}\|_{H^1}^{\frac{1}{2}}\ln^{\frac{1}{4}}\left(2+\|\mathbf{u}\|_{H^1}^2\right)
\bigg(\|\sqrt{\rho}\dot{\mathbf{u}}\|_{L^2}^{\frac{1}{2}}
\|\nabla\mathbf{u}\|_{L^2}^\frac{1}{2}+\|\nabla\mathbf{u}\|_{L^2}\notag\\
&\quad+\frac{1}{2\mu+\lambda}\|\sqrt{\rho}\dot{\mathbf{u}}\|_{L^2}^{\frac{1}{2}}
\|P-P(\tilde{\rho})\|_{L^2}^\frac{1}{2}
+\frac{1}{2\mu+\lambda}\|\nabla\mathbf{u}\|_{L^2}^{\frac{1}{2}}
\|P-P(\tilde{\rho})\|_{L^2}^\frac{1}{2}
+\frac{1}{2\mu+\lambda}
\|P-P(\tilde{\rho})\|_{L^4}\bigg)\notag\\
&\leq C\|\sqrt{\rho}\dot{\mathbf{u}}\|_{L^2}\Big(1+C_0^{\frac{1}{2}}\Big)^{\frac{1}{2}}(1+C_0)^{\frac14}
(1+\|\nabla\mathbf{u}\|_{L^{2}})^{\frac{1}{2}}(1+C_0)^{\frac14}\ln^{\frac{1}{4}}\left(2+\|\nabla\mathbf{u}\|_{L^2}^2\right)
\bigg(\|\sqrt{\rho}\dot{\mathbf{u}}\|_{L^2}^{\frac{1}{2}}
\|\nabla\mathbf{u}\|_{L^2}^\frac{1}{2}\notag\\
&\quad+\|\nabla\mathbf{u}\|_{L^2}
+\frac{C_0^{\frac14}}{2\mu+\lambda}\|\sqrt{\rho}\dot{\mathbf{u}}\|_{L^2}^{\frac{1}{2}}
+\frac{C_0^{\frac14}}{2\mu+\lambda}\|\nabla\mathbf{u}\|_{L^2}^{\frac{1}{2}}
+\frac{1}{2\mu+\lambda}\|P-P(\tilde{\rho})\|_{L^4}\bigg)\notag\\
&\leq \frac{1}{4}\|\sqrt{\rho}\dot{\mathbf{u}}\|_{L^2}^2+C(1+C_0)^4
\big(1+\|\nabla\mathbf{u}\|_{L^{2}}^2\big)
\ln\left(2+\|\nabla\mathbf{u}\|_{L^2}^2\right)\|\nabla\mathbf{u}\|_{L^2}^2
+\frac{C}{(2\mu+\lambda)^4}\|P-P(\tilde{\rho})\|_{L^4}^4.
\end{align}

Substituting \eqref{z3.16}--\eqref{z3.18} into \eqref{z3.9} leads to
\begin{align}\label{z3.19}
&\frac{1}{2}\frac{\mathrm{d}}{\mathrm{d}t}\bigg(\int\left[\mu|\nabla \mathbf{u}|^2+(\mu+\lambda)(\divv \mathbf{u})^2-2(P-P(\tilde{\rho}))\divv\mathbf{u}\right]
\mathrm{d}\mathbf{x}+\mu\int_{\partial\mathbb{R}^2_+}|\mathbf{u}|^2\mathrm{ds}\bigg)
+\frac{1}{2}\int\rho |\dot{\mathbf{u}}|^2\mathrm{d}\mathbf{x}\notag\\
&\leq C(1+C_0)^4\big(1+\|\nabla\mathbf{u}\|_{L^{2}}^2\big)
\ln\big(2+\|\nabla\mathbf{u}\|_{L^2}^2\big)\|\nabla\mathbf{u}\|_{L^2}^2
+\frac{C}{(2\mu+\lambda)^2}\|P-P(\tilde{\rho})\|_{L^4}^4.
\end{align}
Now we define an auxiliary functional $B_1(t)$ as
\begin{align}\label{z3.20}
B_1(t)=&\int\bigg(\frac{1}{2}\rho |\mathbf{u}|^2+G(\rho)\bigg)\mathrm{d}\mathbf{x}
+\frac{\mu}{2}\int_{\partial\mathbb{R}^2_+}|\mathbf{u}|^2\mathrm{ds}\notag\\
&+\frac{1}{2}\int\left[\mu|\nabla \mathbf{u}|^2+(\mu+\lambda)(\divv \mathbf{u})^2-2(P-P(\tilde{\rho}))\divv\mathbf{u}\right]\mathrm{d}\mathbf{x}.
\end{align}
This combined with \eqref{z1.10} and \eqref{z3.1} indicates that there exists $D_1=D_1(\tilde{\rho},\hat{\rho})>0$ such that
\begin{align}\label{z3.21}
B_1(t)\thicksim\int\bigg(\frac{1}{2}\rho |\mathbf{u}|^2+G(\rho)\bigg)\mathrm{d}\mathbf{x}+\frac{1}{2}\int\left[\mu|\nabla \mathbf{u}|^2+(\mu+\lambda)(\divv \mathbf{u})^2\right]\mathrm{d}\mathbf{x}
+\frac{\mu}{2}\int_{\partial\mathbb{R}^2_+}|\mathbf{u}|^2\mathrm{ds}
\end{align}
provided $\lambda\ge D_1$.

Setting
\begin{align}\label{z3.22}
f_1(t)\triangleq2+B_1(t),~~g_1(t)\triangleq(1+C_0)^4\|\nabla\mathbf{u}\|_{L^2}^2+\frac{1}{(2\mu+\lambda)^2}\|P-P(\tilde{\rho})\|_{L^4}^4.
\end{align}
Adding \eqref{z3.6} into \eqref{z3.19}, we thus deduce from \eqref{z3.21} that
\begin{equation*}
    f'_{1}(t)\leq Cg_{1}(t)f_{1}(t)\ln f_{1}(t),
\end{equation*}
which yields that
\begin{equation}\label{z3.23}
    \big(\ln f_1(t)\big)'\leq Cg_1(t)\ln f_1(t).
\end{equation}
Hence, using Gronwall's inequality, there is a positive constant $D_2=D_2(\tilde{\rho},\hat{\rho},a,\gamma,\mu)\geq D_1$ such that
\begin{equation}\label{z3.24}
    \sup_{0\leq t\leq T}\int\left[\mu|\nabla \mathbf{u}|^2+(\mu+\lambda)(\divv \mathbf{u})^2\right]\mathrm{d}\mathbf{x}\leq(2+M)^{\mathrm{e}^{D_2(1+C_0)^5}}
\end{equation}
provided $\lambda\geq D_2$.
Integrating \eqref{z3.19} with respect to $t$ over $(0,T)$ and taking advantage of \eqref{z3.1}, \eqref{z3.3}, and \eqref{z3.24}, we conclude that
\begin{align}\label{z3.25}
\int_0^T\|\sqrt{\rho}{\dot{\mathbf{u}}}\|_{L^2}^2\mathrm{d}t
&\leq CM+C(1+C_0)^4(2+M)^{\mathrm{e}^{D_2(1+C_0)^5}}
\ln\bigg\{(2+M)^{\mathrm{e}^{\frac{3}{2}D_2(1+C_0)^5}}\bigg\}C_0+C_0\notag\\
&\leq(2+M)^{e^{\frac{7}{4}D_2(1+C_0)^5}}
\end{align}
provided that $\lambda$ satisfies \eqref{lam} with $D\geq D_2$, which along with \eqref{z3.24} implies the desired \eqref{z3.7}.
\end{proof}

Next, we give the bound of $\frac{1}{2\mu+\lambda}\int_{0}^{T}\|P-P(\tilde{\rho})\|_{L^4}^4\mathrm{d}t$.
\begin{lemma}\label{l3.3}
Let \eqref{z3.1} be satisfied, then it holds that
\begin{align}\label{z3.26}
\frac{1}{2\mu+\lambda}\int_{0}^{T}\|P-P(\tilde{\rho})\|_{L^4}^4\mathrm{d}t\le (2+M)^{e^{3D_2(1+C_0)^5}}
\end{align}
provided that $\lambda$ satisfies \eqref{lam} with $D\geq 2D_2$.
\end{lemma}
\begin{proof}
It follows from $\eqref{a1}_1$ and $P(\rho)=a\rho^\gamma$ that
\begin{equation}\label{z3.27}
  (P-P(\tilde{\rho}))_t+\mathbf{u}\cdot\nabla(P-P(\tilde{\rho}))+\gamma (P-P(\tilde{\rho}))\divv\mathbf{u}+\gamma P(\tilde{\rho})\divv\mathbf{u}=0.
\end{equation}
Multiplying $\eqref{z3.27}$ by $3(P-P(\tilde{\rho}))^2$ and integrating the resulting equality over $\mathbb{R}^2_+$, we obtain that
\begin{align*}
&\frac{3\gamma-1}{2\mu+\lambda}\left\|P-P(\tilde{\rho})\right\|_{L^4}^4\notag\\
&=-\frac{\mathrm{d}}{\mathrm{d}t}\int(P-P(\tilde{\rho}))^3\mathrm{d}\mathbf{x}
-\frac{3\gamma-1}{2\mu+\lambda}\int(P-P(\tilde{\rho}))^3F\mathrm{d}\mathbf{x}
-3\gamma P(\tilde{\rho})\int(P-P(\tilde{\rho}))^2\divv\mathbf{u}\mathrm{d}\mathbf{x}\notag\\
&\leq-\frac{\mathrm{d}}{\mathrm{d}t}\int(P-P(\tilde{\rho}))^3\mathrm{d}\mathbf{x}+
\frac{3\gamma-1}{2(2\mu+\lambda)}\left\|P-P(\tilde{\rho})\right\|_{L^4}^4+\frac{C}{2\mu+\lambda}\left\|F\right\|_{L^4}^4
+C(2\mu+\lambda)\left\|\divv\mathbf{u}\right\|_{L^2}^2\notag.
\end{align*}
Integrating the above inequality over $(0,T)$, one deduces from \eqref{z3.3}, \eqref{z3.7}, and Lemma $\ref{l3.2}$ that
\begin{align}\label{z3.28}
&\frac{1}{2\mu+\lambda}\int_{0}^{T}\|P-P(\tilde{\rho})\|_{L^{4}}^{4}\mathrm{d}t\notag\\
&\leq C\sup_{0\leq t\leq T}\left\|P-P(\tilde{\rho})\right\|_{L^3}^3+\frac{C}{2\mu+\lambda}\int_0^T\|F\|_{L^2}^2\|\nabla F\|_{L^2}^2\mathrm{d}t+CC_0\notag\\
&\leq C(\tilde{\rho},\hat{\rho},a,\gamma)C_0+\frac{C}{2\mu+\lambda}\int_0^T\left[(2\mu+\lambda)^2
\|\divv\mathbf{u}\|_{L^2}^2+\|P-P(\tilde{\rho})\|_{L^2}^2\right]
\big(\|\sqrt{\rho}\dot{\mathbf{u}}\|_{L^2}^2+\|\nabla\mathbf{u}\|_{L^2}^2\big)\mathrm{d}t\notag\\
&\leq(2+M)^{\mathrm{e}^{3D_{2}(1+C_{0})^5}},
\end{align}
as the desired \eqref{z3.26}.
\end{proof}

Next, motivated by \cite{Hoff95,Hoff95*,HWZ}, we have the following time-weighted estimate on  $\|\sqrt{\rho}{\dot{\mathbf{u}}}\|_{L^2}^2$.
\begin{lemma}\label{l3.4}
Let \eqref{z3.1} be satisfied, then it holds that
\begin{align}\label{z3.29}
\sup_{0\leq t\leq T}(\sigma\|\sqrt{\rho}{\dot{\mathbf{u}}}\|_{L^2}^2)+\int_0^T\big[\mu\sigma\|\nabla \dot{\mathbf{u}}\|_{L^2}^2+(\mu+\lambda)\sigma\|\divf \mathbf{u}\|_{L^2}^2\big]\mathrm{d}t
\leq\exp\Big\{(2+M)^{\mathrm{e}^{4D_{2}(1+C_{0})^5}}\Big\}
\end{align}
provided that $\lambda$ satisfies \eqref{lam} with $D\geq 3D_2$.
\end{lemma}
\begin{proof}
Operating $\sigma\dot{u}^j[\partial/\partial t+\divv({\mathbf{u}}\cdot)]$ on $\eqref{z3.8}^j$, summing all the equalities with respect to $j$, and integrating the resultant over $\mathbb{R}^2_+$, we get from
$\eqref{a1}_1$ and \eqref{a3} that
\begin{align}\label{z3.30}
&\frac{1}{2}\frac{\mathrm{d}}{\mathrm{d}t}\int\sigma\rho|\dot{\mathbf{u}}|^2\mathrm{d}\mathbf{x}
-\frac{\sigma'}{2}\int\rho|\dot{\mathbf{u}}|^2\mathrm{d}\mathbf{x}\notag\\
&=-\sigma\int\dot{u}^j[\partial_j P_{t}+\divv(\mathbf{u}\partial_{j}P)]\mathrm{d}\mathbf{x}
+\mu\sigma\int\dot{u}^j\big[\Delta u_{t}^j+\divv\big( \mathbf{u}\Delta u^{j}\big)\big]\mathrm{d}\mathbf{x}\notag\\
&\quad+(\mu+\lambda)\sigma\int\dot{u}^j[\partial_j\divv\mathbf{u}_t +\divv(\mathbf{u}\partial_j\divv\mathbf{u})]\mathrm{d}\mathbf{x}\triangleq \sum_{i=1}^{3}J_i.
\end{align}
Noting that $\partial\mathbb{R}^2_+$ is flat and $\mathbf{n}=(0,-1)$, a straightforward calculation shows that
\begin{equation*}
u^2=\partial_1u^2=\dot{u}^2=\mathbf{u}\cdot\mathbf{n}=\dot{\mathbf{u}}\cdot\mathbf{n}=\dot{\mathbf{u}}\cdot\nabla\mathbf{u}\cdot\mathbf{n}=0 \ \ \text{on} \ \partial\mathbb{R}^2_+.
\end{equation*}

Integration by parts together with Cauchy--Schwarz inequality gives that
\begin{align}\label{z3.31}
J_{1}&=\sigma\int P_{t}\divv\dot{\mathbf{u}}\mathrm{d}\mathbf{x}-\sigma\int\dot{\mathbf{u}}\cdot\nabla\divv(P\mathbf{u})\mathrm{d}\mathbf{x}
+\sigma\int\dot{u}^j\divv(P\partial_{j}\mathbf{u})\mathrm{d}\mathbf{x}\notag\\
&=\sigma\int \big(P_{t}+\divv(P\mathbf{u})\big)\divv\dot{\mathbf{u}}\mathrm{d}\mathbf{x}
+\sigma\int\dot{\mathbf{u}}\cdot\nabla\mathbf{u}\cdot\nabla P\mathrm{d}\mathbf{x}
+\sigma\int P\dot{\mathbf{u}}\cdot\nabla\divv\mathbf{u}\mathrm{d}\mathbf{x}\notag\\
&=-\sigma\int(\gamma-1)P\divv\mathbf{u}\divv\dot{\mathbf{u}}\mathrm{d}\mathbf{x}
-\sigma\int P\nabla\dot{\mathbf{u}}:\nabla\mathbf{u} \mathrm{d}\mathbf{x}
\notag\\
&\leq \frac{\mu\sigma}{16}\|\nabla\dot{\mathbf{u}}\|_{L^2}^2+C\sigma\|\nabla\mathbf{u}\|_{L^2}^2.
\end{align}
According to Lemma $\ref{E0}$, one gets that
\begin{align*}
J_2&=\mu\sigma\int\dot{u}^j\big[\Delta \dot{u}^j-\Delta(\mathbf{u}\cdot\nabla u^j)+\divv\big( \mathbf{u}\Delta u^{j}\big)\big]\mathrm{d}\mathbf{x}\notag\\
&=\mu\sigma\int\big[-|\nabla\dot{\mathbf{u}}|^2+\dot{u}^j_i(u^ku_k^j)_i
-\dot{u}^j_i(u^ku^j_i)_k-\dot{u}^j(u^k_iu^j_i)_k\big]\mathrm{d}\mathbf{x}\notag\\
&\quad+\mu\sigma\int_{\partial\mathbb{R}^2_+}\big[\dot{u}^j\dot{u}^j_in^i-\dot{u}^j(u^ku_k^j)_in^i
+\dot{u}^j(u^ku^j_i)_kn^i\big]\mathrm{ds}\notag\\
&=\mu\sigma\int\big[-|\nabla\dot{\mathbf{u}}|^2+\dot{u}^j_i(u^ku_k^j)_i
-\dot{u}^j_i(u^ku^j_i)_k+\dot{u}_k^j(u^k_iu^j_i)\big]\mathrm{d}\mathbf{x}\notag\\
&\quad+\mu\sigma\int_{\partial\mathbb{R}^2_+}\big[\dot{u}^j\dot{u}^j_in^i-\dot{u}^j(u^ku_k^j)_in^i
+\dot{u}^j(u^ku^j_i)_kn^i-\dot{u}^j(u^k_iu^j_i)n^k\big]\mathrm{ds}\notag\\
&=\mu\sigma\int\big[-|\nabla\dot{\mathbf{u}}|^2+\dot{u}^j_i(u^ku_k^j)_i
-\dot{u}^j_i(u^ku^j_i)_k+\dot{u}_k^j(u^k_iu^j_i)\big]\mathrm{d}\mathbf{x}\notag\\
&\quad-\mu\sigma\int_{\partial\mathbb{R}^2_+}|\dot{\mathbf{u}}|^2\mathrm{ds}
-\mu\sigma\int_{\partial\mathbb{R}^2_+}(\dot{\mathbf{u}}\cdot\mathbf{u})\partial_1u^1\mathrm{ds}\notag\\
&\leq\mu\sigma\int\big[-|\nabla\dot{\mathbf{u}}|^2+\dot{u}^j_i(u^ku_k^j)_i
-\dot{u}^j_i(u^ku^j_i)_k+\dot{u}_k^j(u^k_iu^j_i)\big]\mathrm{d}\mathbf{x}
-\mu\sigma\int_{\partial\mathbb{R}^2_+}(\dot{\mathbf{u}}\cdot\mathbf{u})\partial_1u^1\mathrm{ds}.
\end{align*}
To handle the boundary integral of the above inequality, we apply the fact that,
for $h\in(C^1\cap W^{1,1})(\overline{\mathbb{R}^2_+})$,
\begin{align*}
&\int_{\partial\mathbb{R}^2_+}h(\mathbf{x})\mathrm{ds}=\int_{\mathbb{R}^2_+\cap\{0\leq x_2\leq1\}}\partial_2[(x_2-1)h(\mathbf{x})]\mathrm{d}\mathbf{x}=\int_{\mathbb{R}^2_+\cap\{0\leq x_2\leq1\}}[h(\mathbf{x})+(x_2-1)h_{x_2}(\mathbf{x})]\mathrm{d}\mathbf{x}.
\end{align*}
Then, integrating by parts in the $x_1$ direction, we obtain the upper bound
\begin{equation*}
  \mu\sigma\int\big(|\mathbf{u}||\nabla\mathbf{u}||\dot{\mathbf{u}}|+|\nabla\mathbf{u}|^2|\dot{\mathbf{u}}|
  +|\mathbf{u}||\nabla\mathbf{u}||\nabla\dot{\mathbf{u}}|\big)\mathrm{d}\mathbf{x}.
\end{equation*}
This combined with \eqref{z3.13} and the following inequality
\begin{align}\label{z3.36}
\|\dot{\mathbf{u}}\|_{L^2}\leq C\|\sqrt{\rho}\dot{\mathbf{u}}\|_{L^2}
+CC_0^{\frac12}\|\nabla\dot{\mathbf{u}}\|_{L^2},
\end{align}
which can be proved similarly to \eqref{z3.13} by replacing $\mathbf{u}$ with $\dot{\mathbf{u}}$, indicates that
\begin{align}\label{z3.32}
J_2&\leq-\frac{3\mu\sigma}{4}\|\nabla\dot{\mathbf{u}}\|_{L^2}^2+CC_0\sigma\|\nabla\mathbf{u}\|_{L^4}^4
+C C_0\sigma \|\nabla\mathbf{u}\|_{L^2}^2\big(1+\|\nabla\mathbf{u}\|_{L^2}^2\big)
+C\sigma\|\sqrt{\rho}\dot{\mathbf{u}}\|_{L^2}^2\notag\\
&\leq-\frac{3\mu\sigma}{4}\|\nabla\dot{\mathbf{u}}\|_{L^2}^2+C(1+C_0)
\sigma\big(1+\|\nabla\mathbf{u}\|_{L^2}^2\big) \big(1+\|\sqrt{\rho}\dot{\mathbf{u}}\|_{L^2}^2+\|\nabla\mathbf{u}\|_{L^2}^2\big)\notag\\
&\quad+\frac{CC_0^2\sigma}{(2\mu+\lambda)^4}
\big(\|\sqrt{\rho}\dot{\mathbf{u}}\|_{L^2}^2+\|\nabla\mathbf{u}\|_{L^2}^2\big)
+\frac{CC_0}{(2\mu+\lambda)^4}\|P-P(\tilde{\rho})\|_{L^4}^4,
\end{align}
due to $0\leq\sigma, \sigma'\leq1$ for $t>0$.
Moreover, in view of \eqref{a3} and \eqref{a4}, we have the following decomposition
\begin{align}\label{z3.33}
J_3&=\sigma(\mu+\lambda)\int\dot{u}^j[\partial_j\divv\mathbf{u}_t+ \divv(\mathbf{u}\partial_j\divv\mathbf{u})]\mathrm{d}\mathbf{x}\notag\\
&=\sigma(\mu+\lambda)\int\dot{u}^j[\partial_j\divv\mathbf{u}_t+ \partial_j\divv(\mathbf{u}\divv\mathbf{u})-\divv(\partial_j\mathbf{u}\divv\mathbf{u})]\mathrm{d}\mathbf{x}\notag\\
&=-\sigma(\mu+\lambda)\int\divv\dot{\mathbf{u}}[\divv\mathbf{u}_t+\divv(\mathbf{u}\divv\mathbf{u})]\mathrm{d}\mathbf{x}-\sigma(\mu+\lambda)\int\dot{u}^j\divv(\partial_j\mathbf{u}\divv\mathbf{u})\mathrm{d}\mathbf{x}\notag\\
&=-\sigma(\mu+\lambda)\int(\divv\mathbf{u}_t+\mathbf{u}\cdot\nabla\divv\mathbf{u}+\nabla\mathbf{u}:\nabla\mathbf{u})
[\divv\mathbf{u}_t+\mathbf{u}\cdot\nabla\divv\mathbf{u}+(\divv\mathbf{u})^2]\mathrm{d}\mathbf{x}
\notag\\&\quad
-\sigma(\mu+\lambda)\int\dot{u}^j\divv(\partial_j\mathbf{u}\divv\mathbf{u})\mathrm{d}\mathbf{x}\notag\\
&=-\sigma(\mu+\lambda)\int\big[(\divf\mathbf{u})^2
+\divf\mathbf{u}(\divv\mathbf{u})^2+\divf\mathbf{u}\nabla\mathbf{u}:\nabla\mathbf{u}
+(\divv\mathbf{u})^2\nabla\mathbf{u}:\nabla\mathbf{u}
-\partial_j\mathbf{u}\cdot\nabla\dot{u}^j\divv\mathbf{u}\big]\mathrm{d}\mathbf{x}\notag\\
&\triangleq-\sigma(\mu+\lambda)\|\divf\mathbf{u}\|_{L^2}^2+\sum_{i=1}^4J_{3i},
\end{align}
where we define
\begin{equation*}
  \divf\mathbf{u}\triangleq\divv\mathbf{u}_t+\mathbf{u}\cdot\nabla\divv\mathbf{u}.
\end{equation*}

Our next goal is to bound each $J_{3i}$. One infers from Lemmas $\ref{GN}$, $\ref{E0}$, and H\"older's inequality that
\begin{align}\label{z3.34}
J_{31}&=-\frac{\sigma(\mu+\lambda)}{(2\mu+\lambda)^2}\int\divf\mathbf{u}(F+P-P(\tilde{\rho}))^2\mathrm{d}\mathbf{x}\notag\\
&\leq\frac{C\sigma(\mu+\lambda)}{(2\mu+\lambda)^2}\|\divf\mathbf{u}\|_{L^2}
\big(\|F\|_{L^4}^2+\|P-P(\tilde{\rho})\|_{L^4}^2\big)\notag\\
&\leq\frac{C\sigma(\mu+\lambda)}{(2\mu+\lambda)^2}\|\divf\mathbf{u}\|_{L^2}\big(\|F\|_{L^2}\|\nabla F\|_{L^2}+\|P-P(\tilde{\rho})\|_{L^4}^2\big)\notag\\
&\leq\frac{C\sigma(\mu+\lambda)}{(2\mu+\lambda)^2}\|\divf\mathbf{u}\|_{L^2}\left[\big((2\mu+\lambda)
\|\divv\mathbf{u}\|_{L^2}+\|P-P(\tilde{\rho})\|_{L^2}\big)
(\|\sqrt{\rho}\dot{\mathbf{u}}\|_{L^2}+\|\nabla\mathbf{u}\|_{L^2})+\|P-P(\tilde{\rho})\|_{L^4}^2\right]\notag\\
&\leq\frac{\sigma(\mu+\lambda)}{16}\|\divf\mathbf{u}\|_{L^2}^2
+C\sigma\big(C_0+\|\nabla\mathbf{u}\|_{L^2}^2\big)
\big(\|\sqrt{\rho}\dot{\mathbf{u}}\|_{L^2}^2+\|\nabla\mathbf{u}\|_{L^2}^2\big)
+\frac{C}{(2\mu+\lambda)^3}\|P-P(\tilde{\rho})\|_{L^4}^4.
\end{align}
Using the Hodge-type decomposition, Lemma $\ref{E0}$, and H\"older's inequality, one gets that
\begin{align}\label{z3.35}
J_{32}
&=-\sigma(\mu+\lambda)\int\divf\mathbf{u}(\mathcal{Q}\mathbf{u}+\mathcal{P}\mathbf{u})_j^i(\mathcal{Q}\mathbf{u}
+\mathcal{P}\mathbf{u})_i^j\mathrm{d}\mathbf{x}\notag\\
&=-\sigma(\mu+\lambda)\int\divf\mathbf{u}\big[(\mathcal{Q}\mathbf{u})_j^i(\mathcal{Q}\mathbf{u})_i^j+
(\mathcal{Q}\mathbf{u})_j^i(\mathcal{P}\mathbf{u})_i^j+(\mathcal{P}\mathbf{u})_j^i(\mathcal{Q}\mathbf{u})_i^j
+(\mathcal{P}\mathbf{u})_j^i(\mathcal{P}\mathbf{u})_i^j\big]\mathrm{d}\mathbf{x}\notag\\
&\leq-\sigma(\mu+\lambda)\int\divf\mathbf{u}(\mathcal{P}\mathbf{u})_j^i(\mathcal{P}\mathbf{u})_i^j\mathrm{d}\mathbf{x}
+\frac{\sigma(\mu+\lambda)}{16}\|\divf\mathbf{u}\|_{L^2}^2+C\sigma(\mu+\lambda)\int|\nabla\mathbf{u}|^2
|\nabla\mathcal{Q}\mathbf{u}|^2\mathrm{d}\mathbf{x}\notag\\
&\triangleq \tilde{J}_{32}+\frac{\sigma(\mu+\lambda)}{16}\|\divf\mathbf{u}\|_{L^2}^2+C\sigma(\mu+\lambda)\int|\nabla\mathbf{u}|^2
|\nabla\mathcal{Q}\mathbf{u}|^2\mathrm{d}\mathbf{x}\notag\\
&\leq\tilde{J}_{32}+\frac{\sigma(\mu+\lambda)}{16}\|\divf\mathbf{u}\|_{L^2}^2+C\sigma\|\nabla\mathbf{u}\|_{L^4}^4
+C\sigma(\mu+\lambda)^2\|\divv\mathbf{u}\|_{L^4}^4\notag\\
&\leq\tilde{J}_{32}+\frac{\sigma(\mu+\lambda)}{16}\|\divf\mathbf{u}\|_{L^2}^2
+C\sigma\big(C_0+\|\nabla\mathbf{u}\|_{L^2}^2\big)
\big(\|\sqrt{\rho}\dot{\mathbf{u}}\|_{L^2}^2+\|\nabla\mathbf{u}\|_{L^2}^2\big)
+\frac{C}{(2\mu+\lambda)^2}\|P-P(\tilde{\rho})\|_{L^4}^4\notag\\&\quad
+\frac{CC_0\sigma}{(2\mu+\lambda)^2}
\big(\|\sqrt{\rho}\dot{\mathbf{u}}\|_{L^2}^2+\|\nabla\mathbf{u}\|_{L^2}^2\big),
\end{align}
where we observe that
\begin{align}
 \tilde{J}_{32}
 &=-\sigma(\mu+\lambda)\int(\mathbf{u}\cdot\nabla\divv\mathbf{u})(\mathcal{P}\mathbf{u})_j^i(\mathcal{P}\mathbf{u})_i^j\mathrm{d}\mathbf{x}
 -\sigma(\mu+\lambda)\int\divv\mathbf{u}_t(\mathcal{P}\mathbf{u})_j^i(\mathcal{P}\mathbf{u})_i^j\mathrm{d}\mathbf{x}\notag\\
 &=-\sigma(\mu+\lambda)\int(\mathbf{u}\cdot\nabla\divv\mathbf{u})(\mathcal{P}\mathbf{u})_j^i(\mathcal{P}\mathbf{u})_i^j\mathrm{d}\mathbf{x}
 -\frac{\sigma(\mu+\lambda)}{2\mu+\lambda}\int(P-P(\tilde{\rho}))_t(\mathcal{P}\mathbf{u})_j^i(\mathcal{P}\mathbf{u})_i^j\mathrm{d}\mathbf{x}\notag\\&\quad-\frac{\sigma(\mu+\lambda)}{2\mu+\lambda}\int F_t(\mathcal{P}\mathbf{u})_j^i(\mathcal{P}\mathbf{u})_i^j\mathrm{d}\mathbf{x}
 \triangleq J_{321}+J_{322}+J_{323}.\notag
\end{align}

By \eqref{z3.14}, \eqref{z3.27}, Lemma $\ref{GN}$, and Lemma $\ref{E0}$, we conclude that
\begin{align}\label{z3.37}
J_{321}&=-\frac{\sigma(\mu+\lambda)}{2\mu+\lambda}\int\mathbf{u}\cdot\nabla(F+P-P(\tilde{\rho}))(\mathcal{P}\mathbf{u})_{j}^{i}(\mathcal{P}\mathbf{u})_{i}^{j}\mathrm{d}\mathbf{x}\notag\\
&\leq C\sigma\|\mathbf{u}\|_{L^4}\|\nabla F\|_{L^2}\|\nabla\mathcal{P}\mathbf{u}\|_{L^8}^2+\frac{\sigma(\mu+\lambda)}{2\mu+\lambda}\int(P-P(\tilde{\rho}))(\mathcal{P}\mathbf{u})_i^j\mathbf{u}\cdot\nabla(\mathcal{P}\mathbf{u})_j^i\mathrm{d}\mathbf{x}
\notag\\
&\quad+\frac{\sigma(\mu+\lambda)}{2\mu+\lambda}\int(P-P(\tilde{\rho}))(\mathcal{P}\mathbf{u})_{j}^{i}\mathbf{u}\cdot\nabla(\mathcal{P}\mathbf{u})_{i}^{j}\mathrm{d}\mathbf{x}
+\frac{\sigma(\mu+\lambda)}{2\mu+\lambda}\int(P-P(\tilde{\rho}))(\mathcal{P}\mathbf{u})_{j}^{i}(\mathcal{P}\mathbf{u})_{i}^{j}\divv\mathbf{u}\mathrm{d}\mathbf{x}\notag\\
&\leq C\sigma\Big(1+C_0^{\frac14}+\|\nabla\mathbf{u}\|_{L^{2}}^{2}\Big)
\big(1+\|\sqrt{\rho}\dot{\mathbf{u}}\|_{L^{2}}^{2}
+\|\nabla\mathbf{u}\|_{L^{2}}^{2}\big)
\big(\|\sqrt{\rho}\dot{\mathbf{u}}\|_{L^{2}}^{2}
+\|\nabla\mathbf{u}\|_{L^{2}}^{2}\big)
\notag\\&\quad+C\sigma\|P-P(\tilde{\rho})\|_{L^\infty}\|\nabla\mathbf{u}\|_{L^2}\|\nabla\mathcal{P}\mathbf{u}\|_{L^4}^2
+C\sigma\|P-P(\tilde{\rho})\|_{L^\infty}\|\nabla\mathcal{P}\mathbf{u}\|_{L^4}
\|\mathbf{u}\|_{L^4}\|\nabla^2\mathcal{P}\mathbf{u}\|_{L^2}\notag\\
&\leq
C\sigma\Big(1+C_0^{\frac14}+\|\nabla\mathbf{u}\|_{L^{2}}^{2}\Big)
\big(1+\|\sqrt{\rho}\dot{\mathbf{u}}\|_{L^{2}}^{2}
+\|\nabla\mathbf{u}\|_{L^{2}}^{2}\big)
\big(\|\sqrt{\rho}\dot{\mathbf{u}}\|_{L^{2}}^{2}
+\|\nabla\mathbf{u}\|_{L^{2}}^{2}\big),
\\
J_{322}&=-\frac{\sigma(\mu+\lambda)}{2\mu+\lambda}\int(P-P(\tilde{\rho}))_t
(\mathcal{P}\mathbf{u})_{j}^{i}(\mathcal{P}\mathbf{u})_{i}^{j}\mathrm{d}\mathbf{x}\notag\\
&=\frac{\sigma(\mu+\lambda)}{2\mu+\lambda}\int\big(\mathbf{u}\cdot\nabla(P-P(\tilde{\rho}))+\gamma (P-P(\tilde{\rho}))\divv\mathbf{u}+\gamma P(\tilde{\rho})\divv\mathbf{u}\big)
(\mathcal{P}\mathbf{u})_{j}^{i}(\mathcal{P}\mathbf{u})_{i}^{j}\mathrm{d}\mathbf{x}\notag\\
&\leq C\sigma\Big(1+C_0^{\frac14}+\|\nabla\mathbf{u}\|_{L^{2}}^{2}\Big)
\big(1+\|\sqrt{\rho}\dot{\mathbf{u}}\|_{L^{2}}^{2}+\|\nabla\mathbf{u}\|_{L^{2}}^{2}\big)
\big(\|\sqrt{\rho}\dot{\mathbf{u}}\|_{L^{2}}^{2}+\|\nabla\mathbf{u}\|_{L^{2}}^{2}\big).\label{z3.38}
\end{align}
Owing to $\divv(\mathcal{P}\mathbf{u})=0$, one deduces from \eqref{a3}, \eqref{z3.14}, \eqref{z3.36}, and Lemma $\ref{E0}$ that
\begin{align}\label{z3.39}
&J_{323}=-\frac{\sigma(\mu+\lambda)}{2\mu+\lambda}\int F_t
(\mathcal{P}\mathbf{u})_{j}^{i}(\mathcal{P}\mathbf{u})_{i}^{j}\mathrm{d}\mathbf{x}\notag\\
&=-\frac{\mu+\lambda}{2\mu+\lambda}\frac{\mathrm{d}}{\mathrm{d}t}\int\sigma F
(\mathcal{P}\mathbf{u})_{j}^{i}(\mathcal{P}\mathbf{u})_{i}^{j}\mathrm{d}\mathbf{x}
+\frac{\mu+\lambda}{2\mu+\lambda}\int \sigma'F
(\mathcal{P}\mathbf{u})_{j}^{i}(\mathcal{P}\mathbf{u})_{i}^{j}\mathrm{d}\mathbf{x}
+\frac{\mu+\lambda}{2\mu+\lambda}\int \sigma F
(\mathcal{P}\mathbf{u})_{jt}^{i}(\mathcal{P}\mathbf{u})_{i}^{j}\mathrm{d}\mathbf{x}
\notag\\
&\quad+\frac{\mu+\lambda}{2\mu+\lambda}\int \sigma F
(\mathcal{P}\mathbf{u})_{j}^{i}(\mathcal{P}\mathbf{u})_{it}^{j}\mathrm{d}\mathbf{x}\notag\\
&=\frac{\mu+\lambda}{2\mu+\lambda}\frac{\mathrm{d}}{\mathrm{d}t}\bigg(\int\sigma F_j
(\mathcal{P}\mathbf{u})^{i}(\mathcal{P}\mathbf{u})_{i}^{j}\mathrm{d}\mathbf{x}
-\int_{\partial\mathbb{R}^2_+}\sigma F
(\mathcal{P}\mathbf{u})^{i}(\mathcal{P}\mathbf{u})_{i}^{j}n^j\mathrm{ds}\bigg)
-\frac{\mu+\lambda}{2\mu+\lambda}\int \sigma'F_j
(\mathcal{P}\mathbf{u})^{i}(\mathcal{P}\mathbf{u})_{i}^{j}\mathrm{d}\mathbf{x}
\notag\\&\quad
-\frac{\mu+\lambda}{2\mu+\lambda}\int \sigma F_i
(\mathcal{P}\mathbf{u})_{jt}^{i}(\mathcal{P}\mathbf{u})^{j}\mathrm{d}\mathbf{x}
-\frac{\mu+\lambda}{2\mu+\lambda}\int \sigma F_j
(\mathcal{P}\mathbf{u})^{i}(\mathcal{P}\mathbf{u})_{it}^{j}\mathrm{d}\mathbf{x}
+\frac{\mu+\lambda}{2\mu+\lambda}\int_{\partial\mathbb{R}^2_+}\sigma' F(\mathcal{P}\mathbf{u})^{i}(\mathcal{P}\mathbf{u})_{i}^{j}n^j\mathrm{ds}
\notag\\&\quad
+\frac{\mu+\lambda}{2\mu+\lambda}\int_{\partial\mathbb{R}^2_+}\sigma F(\mathcal{P}\mathbf{u})_{jt}^{i}(\mathcal{P}\mathbf{u})^{j}n^i\mathrm{ds}
+\frac{\mu+\lambda}{2\mu+\lambda}\int_{\partial\mathbb{R}^2_+}\sigma F(\mathcal{P}\mathbf{u})^{i}(\mathcal{P}\mathbf{u})_{it}^{j}n^j\mathrm{ds}
\notag\\
&\triangleq
\frac{\mu+\lambda}{2\mu+\lambda}\frac{\mathrm{d}}{\mathrm{d}t}\bigg(\int\sigma F_j
(\mathcal{P}\mathbf{u})^{i}(\mathcal{P}\mathbf{u})_{i}^{j}\mathrm{d}\mathbf{x}
-\int_{\partial\mathbb{R}^2_+}\sigma F
(\mathcal{P}\mathbf{u})^{i}(\mathcal{P}\mathbf{u})_{i}^{j}n^j\mathrm{ds}\bigg)
-\frac{\mu+\lambda}{2\mu+\lambda}\int \sigma'F_j
(\mathcal{P}\mathbf{u})^{i}(\mathcal{P}\mathbf{u})_{i}^{j}\mathrm{d}\mathbf{x}
\notag\\&\quad
-\frac{\mu+\lambda}{2\mu+\lambda}\int \sigma F_i
\big(\mathcal{P}(\dot{\mathbf{u}}-\mathbf{u}\cdot\nabla\mathbf{u})\big)_j^{i}(\mathcal{P}\mathbf{u})^{j}\mathrm{d}\mathbf{x}
-\frac{\mu+\lambda}{2\mu+\lambda}\int \sigma F_j
(\mathcal{P}\mathbf{u})^{i}\big(\mathcal{P}(\dot{\mathbf{u}}-\mathbf{u}\cdot\nabla\mathbf{u})\big)_i^{j}\mathrm{d}\mathbf{x}+\sum_{i=1}^3\tilde{B}_{i}\notag\\
&\leq\frac{\mu+\lambda}{2\mu+\lambda}\frac{\mathrm{d}}{\mathrm{d}t}\bigg(\int\sigma F_j
(\mathcal{P}\mathbf{u})^{i}(\mathcal{P}\mathbf{u})_{i}^{j}\mathrm{d}\mathbf{x}
-\int_{\partial\mathbb{R}^2_+}\sigma F
(\mathcal{P}\mathbf{u})^{i}(\mathcal{P}\mathbf{u})_{i}^{j}n^j\mathrm{ds}\bigg)
+C\|\nabla F\|_{L^2}\|\mathbf{u}\|_{L^4}\|\nabla\mathcal{P}\mathbf{u}\|_{L^4}
\notag\\&\quad
+C\sigma\|\nabla F\|_{L^2}\|\mathcal{P}\mathbf{u}\|_{L^\infty}
\big(\|\nabla \dot{\mathbf{u}}\|_{L^2}+\|\nabla\mathbf{u}\|_{L^4}^2
+\|\mathbf{u}\|_{L^\infty}\|\nabla^2\mathcal{P}\mathbf{u}\|_{L^2}\big)
+\sum_{i=1}^3\tilde{B}_{i}
\notag\\
&\leq\frac{\mu+\lambda}{2\mu+\lambda}\frac{\mathrm{d}}{\mathrm{d}t}\bigg(\int\sigma F_j
(\mathcal{P}\mathbf{u})^{i}(\mathcal{P}\mathbf{u})_{i}^{j}\mathrm{d}\mathbf{x}
-\int_{\partial\mathbb{R}^2_+}\sigma F
(\mathcal{P}\mathbf{u})^{i}(\mathcal{P}\mathbf{u})_{i}^{j}n^j\mathrm{ds}\bigg)
+\frac{\mu\sigma}{16}\|\nabla \dot{\mathbf{u}}\|_{L^2}^2
\notag\\&\quad
+C\big(1+C_0+\|\nabla\mathbf{u}\|_{L^{2}}^{2}\big)
\big(1+\sigma\|\sqrt{\rho}\dot{\mathbf{u}}\|_{L^{2}}^{2}
+\|\nabla\mathbf{u}\|_{L^{2}}^{2}\big)
\big(\|\sqrt{\rho}\dot{\mathbf{u}}\|_{L^{2}}^{2}
+\|\nabla\mathbf{u}\|_{L^{2}}^{2}\big)\notag\\&\quad
+\frac{CC_0^\frac12\sigma}{(2\mu+\lambda)^2}\big(\|\sqrt{\rho}\dot{\mathbf{u}}\|_{L^2}^2
+\|\nabla\mathbf{u}\|_{L^2}^2\big)
+\frac{C}{(2\mu+\lambda)^4}\|P-P(\tilde{\rho})\|_{L^4}^4+\sum_{i=1}^3\tilde{B}_{i},
\end{align}
where we have used
\begin{gather}
\|\mathcal{P}\mathbf{u}\|_{L^\infty}\leq C\|\mathcal{P}\mathbf{u}\|_{L^4}^\frac12\|\nabla\mathcal{P}\mathbf{u}\|_{L^4}^\frac12+C\|\mathcal{P}\mathbf{u}\|_{L^2},\notag\\
\|\dot{\mathbf{u}}\|_{L^4}\leq C\|\dot{\mathbf{u}}\|_{L^2}^{\frac12}\|\nabla\dot{\mathbf{u}}\|_{L^2}^{\frac12}\leq C\|\sqrt{\rho}\dot{\mathbf{u}}\|_{L^2}^\frac12\|\nabla\dot{\mathbf{u}}\|_{L^2}^\frac12+CC_0^{\frac14}\|\nabla\dot{\mathbf{u}}\|_{L^2},\label{z3.40}
\end{gather}
due to Lemma $\ref{GN}$ and \eqref{z3.36}.

To control the boundary integral terms $\tilde{B}_{i}\ (i=1,2,3)$ in \eqref{z3.39}, note that
\begin{align}\label{z3.41}
\|\nabla\mathcal{Q}\mathbf{u}\|_{L^4}\leq C\|\divv\mathbf{u}\|_{L^4}
\leq \frac{C}{2\mu+\lambda}\big(\|F\|_{L^4}+\|P(\rho)-P(\tilde{\rho})\|_{L^4}\big),
\end{align}
and that, using the boundary condition \eqref{a3},
\begin{align}\label{z3.42}
\mathbf{u}\cdot\nabla(\mathcal{P}\mathbf{u})\cdot\mathbf{n}=(\mathcal{P}\mathbf{u})^{i}(\mathcal{P}\mathbf{u})_{i}^{j}n^j
+(\mathcal{Q}\mathbf{u})^{i}(\mathcal{P}\mathbf{u})_{i}^{j}n^j=0 \ \ \text{on}\ \partial\mathbb{R}^2_+,
\end{align}
we thus infer from Lemma $\ref{E0}$ and the divergence theorem that
\begin{align}\label{z3.43}
\tilde{B}_{1}&=\frac{\mu+\lambda}{2\mu+\lambda}\int_{\partial\mathbb{R}^2_+}\sigma' F(\mathcal{ P}\mathbf{u})^{i}(\mathcal{P}\mathbf{u})_{i}^{j}n^j\mathrm{ds}=-\frac{\mu+\lambda}{2\mu+\lambda}\int_{\partial\mathbb{R}^2_+}\sigma' F(\mathcal{ Q}\mathbf{u})^{i}(\mathcal{P}\mathbf{u})_{i}^{j}n^j\mathrm{ds}\notag\\
&=-\frac{\mu+\lambda}{2\mu+\lambda}\int\sigma'\big(\mathcal{Q}\mathbf{u}\cdot\nabla\mathcal{P}\mathbf{u}\cdot\nabla F+F(\mathcal{Q}\mathbf{u})^{i}_j(\mathcal{P}\mathbf{u})_{i}^{j}\big)\mathrm{d}\mathbf{x}\notag\\
&\leq C\|\nabla F\|_{L^2}\|\mathbf{u}\|_{L^4}\|\nabla\mathcal{P}\mathbf{u}\|_{L^4}
+\frac{C(\mu+\lambda)}{(2\mu+\lambda)^2}\big(\|F\|_{L^4}^2+\|P(\rho)-P(\tilde{\rho})\|_{L^4}^2\big)\|\nabla\mathcal{P}\mathbf{u}\|_{L^2}\notag\\
&\leq C\Big(1+C_0^\frac12\Big)\big(1+\|\nabla\mathbf{u}\|_{L^{2}}^{2}\big)
\big(\|\sqrt{\rho}\dot{\mathbf{u}}\|_{L^{2}}^{2}+\|\nabla\mathbf{u}\|_{L^{2}}^{2}\big)
+\frac{C}{(2\mu+\lambda)^2}\|P(\rho)-P(\tilde{\rho})\|_{L^4}^4,\\
\tilde{B}_{2}+\tilde{B}_{3}
&=-\frac{2(\mu+\lambda)}{2\mu+\lambda}\int_{\partial\mathbb{R}^2_+}\sigma F(\mathcal{Q}\mathbf{u})^{i}(\mathcal{P}\mathbf{u})_{it}^{j}n^j\mathrm{ds}\notag\\
&=-\frac{2(\mu+\lambda)}{2\mu+\lambda}\int\sigma\big[\mathcal{Q}\mathbf{u}\cdot\nabla\mathcal{P}(\dot{\mathbf{u}}-\mathbf{u}\cdot\nabla\mathbf{u})\cdot\nabla F+F(\mathcal{Q}\mathbf{u})^{i}_j\big(\mathcal{P}(\dot{\mathbf{u}}-\mathbf{u}\cdot\nabla\mathbf{u})\big)_{i}^{j}\big]\mathrm{d}\mathbf{x}\notag\\
&\leq C\sigma\|\nabla F\|_{L^2}\|\mathcal{Q}\mathbf{u}\|_{L^\infty}
\big(\|\nabla \dot{\mathbf{u}}\|_{L^2}+\|\nabla\mathbf{u}\|_{L^4}^2
+\|\mathbf{u}\|_{L^\infty}\|\nabla^2\mathcal{P}\mathbf{u}\|_{L^2}\big)
\notag\\
&\quad+\frac{C\sigma(\mu+\lambda)}{(2\mu+\lambda)^2}\big(\|F\|_{L^4}^2+\|P(\rho)-P(\tilde{\rho})\|_{L^4}^2\big)
\big(\|\nabla\dot{\mathbf{u}}\|_{L^2}+\|\mathbf{u}\|_{L^4}\|\nabla\mathbf{u}\|_{L^4}\big)\notag\\
&\leq \frac{\mu\sigma}{16}\|\nabla\dot{\mathbf{u}}\|_{L^2}^2
+C\sigma\big(1+C_0+\|\nabla\mathbf{u}\|_{L^{2}}^{2}\big)
\big(1+\|\sqrt{\rho}\dot{\mathbf{u}}\|_{L^{2}}^{2}+\|\nabla\mathbf{u}\|_{L^{2}}^{2}\big)
\big(\|\sqrt{\rho}\dot{\mathbf{u}}\|_{L^{2}}^{2}+\|\nabla\mathbf{u}\|_{L^{2}}^{2}\big)\notag\\
&\quad+\frac{C}{(2\mu+\lambda)^2}\|P-P(\tilde{\rho})\|_{L^4}^4.\label{z3.44}
\end{align}

Furthermore, using integration by parts and performing similar arguments, we get that
\begin{align}\label{z3.45}
J_{33}&=-\sigma(\mu+\lambda)\int(\divv\mathbf{u})^2\nabla\mathbf{u}:\nabla\mathbf{u}\mathrm{d}\mathbf{x}\notag\\
&\leq C\sigma(\mu+\lambda)\int(\divv\mathbf{u})^2\big(|\nabla\mathcal{P}\mathbf{u}|^2+|\nabla\mathcal{Q}\mathbf{u}|^2\big)\mathrm{d}\mathbf{x}\notag\\
&\leq C\sigma(\mu+\lambda)^2\|\divv\mathbf{u}\|_{L^4}^4+C\sigma\|\nabla\mathcal{P}\mathbf{u}\|_{L^4}^4\notag\\
&\leq\frac{C\sigma(\mu+\lambda)^2}{(2\mu+\lambda)^4}\|F+P-P(\tilde{\rho})\|_{L^4}^4
+C\sigma\big(\|\sqrt{\rho}\dot{\mathbf{u}}\|_{L^{2}}^{2}\|\nabla\mathbf{u}\|_{L^{2}}^{2}
+\|\nabla\mathbf{u}\|_{L^{2}}^{4}\big)\notag\\
&\leq C\sigma \big(\|\sqrt{\rho}\dot{\mathbf{u}}\|_{L^2}^2+\|\nabla\mathbf{u}\|_{L^2}^2+C_0\big)
\big(\|\sqrt{\rho}\dot{\mathbf{u}}\|_{L^2}^2+\|\nabla\mathbf{u}\|_{L^2}^2\big)
+\frac{C}{(2\mu+\lambda)^2}\|P-P(\tilde{\rho})\|_{L^4}^4,\\
J_{34}&=\frac{\sigma(\mu+\lambda)}{2\mu+\lambda}\int\partial_j\mathbf{u}\cdot\nabla\dot{u}^j(F+P-P(\tilde{\rho}))\mathrm{d}\mathbf{x}\notag\\
&=\frac{\sigma(\mu+\lambda)}{2\mu+\lambda}\int\big(\mathcal{P}\mathbf{u}+\mathcal{Q}\mathbf{u}\big)_j\cdot\nabla\dot{u}^jF\mathrm{d}\mathbf{x}
+\frac{\sigma(\mu+\lambda)}{2\mu+\lambda}\int\partial_j\mathbf{u}\cdot\nabla\dot{u}^j(P-P(\tilde{\rho}))\mathrm{d}\mathbf{x}\notag\\
&=-\frac{\sigma(\mu+\lambda)}{2\mu+\lambda}\int(\mathcal{P}\mathbf{u})_j\cdot\nabla F\dot{u}^j\mathrm{d}\mathbf{x}
+\frac{\sigma(\mu+\lambda)}{2\mu+\lambda}\int\partial_j\mathbf{u}\cdot\nabla\dot{u}^j(P-P(\tilde{\rho}))\mathrm{d}\mathbf{x}\notag\\
&\quad+\frac{\sigma(\mu+\lambda)}{2\mu+\lambda}\int(\mathcal{Q}\mathbf{u})_j\cdot\nabla\dot{u}^j F\mathrm{d}\mathbf{x}\notag\\
&\leq C\sigma\|\nabla F\|_{L^2}\|\nabla\mathcal{P}\mathbf{u}\|_{L^4}\|\dot{\mathbf{u}}\|_{L^4}
+C\sigma\|\nabla\mathbf{u}\|_{L^2}\|\nabla\dot{\mathbf{u}}\|_{L^2}\|P-P(\tilde{\rho})\|_{L^\infty}\notag\\
&\quad+\frac{C\sigma(\mu+\lambda)}{(2\mu+\lambda)^2}\|\nabla\dot{\mathbf{u}}\|_{L^2}\big(\|F\|_{L^4}^2+\|P(\rho)-P(\tilde{\rho})\|_{L^4}^2\big)\notag\\
&\leq \frac{\mu\sigma}{16}\|\nabla\dot{\mathbf{u}}\|_{L^2}^2
+C\sigma\big(1+C_0+\|\nabla\mathbf{u}\|_{L^{2}}^{2}\big)
\big(\|\sqrt{\rho}\dot{\mathbf{u}}\|_{L^{2}}^{2}+\|\nabla\mathbf{u}\|_{L^{2}}^{2}\big)
\big(1+\|\sqrt{\rho}\dot{\mathbf{u}}\|_{L^{2}}^{2}\big)\notag\\
&\quad+\frac{C}{(2\mu+\lambda)^2}\|P-P(\tilde{\rho})\|_{L^4}^4.\label{z3.46}
\end{align}

Hence, substituting \eqref{z3.31}, \eqref{z3.32}--\eqref{z3.39}, and \eqref{z3.43}--\eqref{z3.46} into \eqref{z3.30}, one derives from \eqref{z3.7} that
\begin{align}\label{z3.47}
&\frac{\mathrm{d}}{\mathrm{d}t}\bigg(\frac{1}{2}\int\sigma\rho|\dot{\mathbf{u}}|^2\mathrm{d}\mathbf{x}
-\frac{\mu+\lambda}{2\mu+\lambda}\int\sigma F_j
(\mathcal{P}\mathbf{u})^{i}(\mathcal{P}\mathbf{u})_{i}^{j}\mathrm{d}\mathbf{x}
+\frac{\mu+\lambda}{2\mu+\lambda}\int_{\partial\mathbb{R}^2_+}\sigma F
(\mathcal{P}\mathbf{u})^{i}(\mathcal{P}\mathbf{u})_{i}^{j}n^j\mathrm{ds}\bigg)
\notag\\&\quad
+\frac{\mu\sigma}{2}\|\nabla\dot{\mathbf{u}}\|_{L^2}^2
+\frac{(\mu+\lambda)\sigma}{2}\|\divf\mathbf{u}\|_{L^2}^2\notag\\
&\leq(2+M)^{e^{3D_2(1+C_0)^5}}\big(1+\sigma\|\sqrt{\rho}\dot{\mathbf{u}}\|_{L^{2}}^{2}
+\|\nabla\mathbf{u}\|_{L^{2}}^{2}\big)
\big(\|\sqrt{\rho}\dot{\mathbf{u}}\|_{L^{2}}^{2}+\|\nabla\mathbf{u}\|_{L^{2}}^{2}\big)
+\frac{C}{(2\mu+\lambda)^2}\|P-P(\tilde{\rho})\|_{L^4}^4,
\end{align}
where we observe that
\begin{align*}
&\bigg|\int\sigma F_j
(\mathcal{P}\mathbf{u})^{i}(\mathcal{P}\mathbf{u})_{i}^{j}\mathrm{d}\mathbf{x}\bigg|
+\bigg|\int_{\partial\mathbb{R}^2_+}\sigma F
(\mathcal{P}\mathbf{u})^{i}(\mathcal{P}\mathbf{u})_{i}^{j}n^j\mathrm{ds}\bigg|\notag\\
&=\bigg|\int\sigma F_j
(\mathcal{P}\mathbf{u})^{i}(\mathcal{P}\mathbf{u})_{i}^{j}\mathrm{d}\mathbf{x}\bigg|
+\bigg|\int_{\partial\mathbb{R}^2_+}\sigma F
(\mathcal{Q}\mathbf{u})^{i}(\mathcal{P}\mathbf{u})_{i}^{j}n^j\mathrm{ds}\bigg|\notag\\
&\leq C\sigma\|\nabla F\|_{L^2}\|\mathbf{u}\|_{L^4}\|\nabla\mathcal{P}\mathbf{u}\|_{L^4}
+\frac{C\sigma}{2\mu+\lambda}\big(\|F\|_{L^4}^2+\|P(\rho)-P(\tilde{\rho})\|_{L^4}^2\big)\|\nabla\mathcal{P}\mathbf{u}\|_{L^2}\notag\\
&\leq
\frac{\sigma}{4}\|\sqrt{\rho}\dot{\mathbf{u}}\|_{L^{2}}^2
+C\sigma\big(1+C_0^2\big)\|\nabla\mathbf{u}\|_{L^{2}}^2
\big(1+\|\nabla\mathbf{u}\|_{L^{2}}^2\big)^2.
\end{align*}
Now we define an auxiliary functional $B_2(t)$ as
\begin{align*}
  B_2(t)&=\int\frac{\sigma}{2}\rho|\dot{\mathbf{u}}|^2\mathrm{d}\mathbf{x}
-\frac{\mu+\lambda}{2\mu+\lambda}\int\sigma F_j
(\mathcal{P}\mathbf{u})^{i}(\mathcal{P}\mathbf{u})_{i}^{j}\mathrm{d}\mathbf{x}
\notag\\
&\quad+\frac{\mu+\lambda}{2\mu+\lambda}\int_{\partial\mathbb{R}^2_+}\sigma F
(\mathcal{P}\mathbf{u})^{i}(\mathcal{P}\mathbf{u})_{i}^{j}n^j\mathrm{ds}+(2+M)^{e^{3D_2(1+C_0)^5}}B_1(t),
\end{align*}
where, by the definition of $B_1(t)$ in \eqref{z3.20}, one sees that
\begin{equation}\label{z3.48}
B_2(t)\thicksim \sigma\|\sqrt{\rho}\dot{\mathbf{u}}\|_{L^{2}}^2+B_1(t).
\end{equation}

Setting
\begin{equation*}
  f_2(t)\triangleq 2+B_2(t),~~g_2(t)\triangleq(2+M)^{e^{3D_2(1+C_0)^5}}\bigg(\|\sqrt{\rho}\dot{\mathbf{u}}\|_{L^{2}}^2+\|\nabla\mathbf{u}\|_{L^{2}}^2
  +\frac{1}{(2\mu+\lambda)^2}\|P-P(\tilde{\rho})\|_{L^4}^4\bigg)
\end{equation*}
and taking the summation
\begin{align*}
(2+M)^{e^{\frac{13}{4}D_2(1+C_0)^5}}\times\big(\eqref{z3.6}+\eqref{z3.19}\big)+\eqref{z3.47},
\end{align*}
we obtain from \eqref{z3.48} that
\begin{align}\label{z3.49}
f'_2(t)\leq g_2(t)f_2(t)
\end{align}
provided that $\lambda$ satisfies \eqref{lam} with $D\geq 3D_2$.
Thus, by Gronwall's inequality and Lemmas \ref{l3.1}\text{--}\ref{l3.3}, one has that
\begin{equation}\label{z3.50}
  \sup_{0\leq t\leq T}\left(\sigma\|\sqrt{\rho}\dot{\mathbf{u}}\|_{L^{2}}^2\right)\leq
  \exp\bigg\{(2+M)^{e^{\frac{7}{2}D_2(1+C_0)^5}}\bigg\}.
\end{equation}
Integrating \eqref{z3.47} with respect to $t$ over $(0,T)$, one infers from \eqref{z3.50}, Lemmas $\ref{E0}$, and $\ref{l3.1}\text{--}\ref{l3.3}$ that
\begin{align*}
&\int_0^T\big[\mu\sigma\|\nabla\dot{\mathbf{u}}\|_{L^2}^2
+(\mu+\lambda)\sigma\|\divf\mathbf{u}\|_{L^2}^2\big]\mathrm{d}t\notag\\
&\leq CM+C(1+C_0)^3(2+M)^{e^{\frac{7}{2}D_2(1+C_0)^5}}
\exp\bigg\{(2+M)^{e^{\frac{7}{2}D_2(1+C_0)^5}}\bigg\}\notag\\
&\leq\exp\bigg\{(2+M)^{e^{\frac{15}{4}D_2(1+C_0)^5}}\bigg\},
\end{align*}
which along with \eqref{z3.50} leads to the desired \eqref{z3.29}.
\end{proof}

Finally, inspired by \cite{DE97}, we derive the upper bound of density.
\begin{lemma}\label{l3.5}
Under the assumption \eqref{z3.1}, it holds that
\begin{align*}
0\leq\rho(\mathbf{x},t)\leq\frac{7}{4}\hat{\rho}~\textit{a.e.}~\mathrm{on}~\mathbb{R}^2_+\times[0,T]
\end{align*}
provided that $\lambda$ satisfies \eqref{lam} with $D\geq 5D_2$.
\end{lemma}
\begin{proof}
Let $\mathbf{y}\in\mathbb{R}^2_+$ and define the corresponding particle path $\mathbf{x}(t)$ by
\begin{align*}
\begin{cases}
\mathbf{\dot{x}}(t,\mathbf{y})=\mathbf{u}(\mathbf{x}(t,\mathbf{y}),\mathbf{y)},\\
\mathbf{\dot{x}}(t_0,\mathbf{y})=\mathbf{y}.
\end{cases}
\end{align*}
Assume that there exists $t_1\leq T$ satisfying $\rho(\mathbf{x}(t_1), t_1) = \frac{7}{4}\hat{\rho}$, we take a minimal value of $t_1$ and then choose a maximal value of $t_0<t_1$ such that $\rho(\mathbf{x}(t_0), t_0)=\frac{3}{2}\hat{\rho}$. Thus, $\rho(\mathbf{x}(t),t)\in[\frac{3}{2}\hat{\rho},\frac{7}{4}\hat{\rho}]$ for $t\in[t_0,t_1]$. We divide
the argument into two cases.

\textbf{Case 1:} If $t_0<t_1\leq1$, one derives from $\eqref{a1}_1$ and \eqref{z1.5} that
\begin{equation*}
(2\mu+\lambda)\frac{\mathrm{d}}{\mathrm{d}t}\ln\rho(\mathbf{x}(t),t)
+P(\rho(\mathbf{x}(t),t))-P(\tilde{\rho})=-F(\mathbf{x}(t),t),
\end{equation*}
where $\frac{\mathrm{d}\rho}{\mathrm{d}t}\triangleq \rho_t+\mathbf{u}\cdot\nabla\rho$. Integrating the above equality from $t_0$ to $t_1$ and abbreviating $\rho(\mathbf{x},t)$ by $\rho(t)$ for
convenience, one gets that
\begin{equation}\label{z3.51}
\ln\rho(\tau)\big|_{t_0}^{t_1}+\frac{1}{2\mu+\lambda}\int_{t_0}^{t_1}\left(P(\rho(\tau))
-P(\tilde{\rho})\right)\mathrm{d}\tau=-\frac{1}{2\mu+\lambda}\int_{t_0}^{t_1}F(\mathbf{x}(\tau),\tau)\mathrm{d}\tau.
\end{equation}
It follows from \eqref{z3.36}, Lemmas $\ref{GN}$, $\ref{E0}$, and $\ref{l3.1}\text{--}\ref{l3.4}$ that
\begin{align}\label{z3.52}
&\int_0^{\sigma(T)}\left\|F(\cdot,t)\right\|_{L^\infty}\mathrm{d}t\leq C\int_0^{\sigma(T)}\left\|F\right\|_{L^2}^{\frac13}\left\|\nabla F\right\|_{L^4}^{\frac23}\mathrm{d}t\notag\\
&\leq C\int_0^{\sigma(T)}\Big((2\mu+\lambda)^{\frac13}
\|\divv\mathbf{u}\|_{L^2}^{\frac13}+\|P-P(\tilde{\rho})\|_{L^2}^{\frac13}\Big)
\Big(\|\sqrt{\rho}\dot{\mathbf{u}}\|_{L^4}^{\frac23}+\|\nabla\mathbf{u}\|_{L^4}^{\frac23}\Big)
\mathrm{d}t\notag\\
&\leq C\sup_{0\leq t\leq T}\Big((2\mu+\lambda)^{\frac13}\|\divv\mathbf{u}\|_{L^2}^{\frac13}
+\|P-P(\tilde{\rho})\|_{L^2}^{\frac13}\Big)
\int_{0}^{\sigma(T)}
\|\dot{\mathbf{u}}\|_{L^{2}}^{\frac{1}{3}}
\|\nabla\dot{\mathbf{u}}\|_{L^{2}}^{\frac{1}{3}}
 \mathrm{d}t\notag\\
&\leq C\Big((2\mu+\lambda)^{\frac{1}{6}}(2+M)^{\frac{1}{6}\mathrm{e}^{{2D_{2}
(1+C_{0})^{5}}}}+C_0^\frac16\Big)
\int_{0}^{\sigma(T)}
\Big(\|\sqrt{\rho}\dot{\mathbf{u}}\|_{L^{2}}^{\frac{1}{3}}\|\nabla\dot{\mathbf{u}}
\|_{L^{2}}^{\frac{1}{3}}+C_0^\frac16\|\nabla\dot{\mathbf{u}}
\|_{L^{2}}^{\frac{2}{3}}\Big)\mathrm{d}t\notag\\
&\leq(2\mu+\lambda)^{\frac16}(2+M)^{\frac12\mathrm{e}^{2D_2(1+C_0)^5}}\bigg[\int_0^{\sigma(T)}\|\sqrt{\rho}\dot{\mathbf{u}}\|_{L^2}^{\frac23}\mathrm{d}t
+\int_0^{\sigma(T)}(t\|\nabla\dot{\mathbf{u}}\|_{L^2}^2)^{\frac13}t^{-\frac13}\mathrm{d}t\bigg]\notag\\
&\leq(2\mu+\lambda)^{\frac{1}{6}}(2+M)^{\frac{1}{2}\mathrm{e}^{2D_{2}(1+C_{0})^{5}}} \Bigg[\bigg(\int_{0}^{\sigma(T)}\|\sqrt{\rho}\dot{\mathbf{u}}\|_{L^{2}}^{2}\mathrm{d}t\bigg)^{\frac{1}{3}}\bigg(\int_{0}^{\sigma(T)}1\mathrm{d}t\bigg)^{\frac{2}{3}}
\notag\\&\quad+\bigg(\int_0^{\sigma(T)}t\|\nabla\dot{\mathbf{u}}\|_{L^2}^2
\mathrm{d}t\bigg)^{\frac13}\bigg(\int_0^{\sigma(T)}
t^{-\frac12}\mathrm{d}t\bigg)^{\frac23}\Bigg]\notag\\
&\leq(2\mu+\lambda)^{\frac{1}{6}}\exp
\bigg\{(2+M)^{\mathrm{e}^{\frac{9}{2}D_{2}(1+C_{0})^{5}}}\bigg\}.
\end{align}
Note that $\rho(t)$ takes values in $[\frac{3}{2}\hat{\rho},\frac{7}{4}\hat{\rho}]\subset [\hat{\rho},2\hat{\rho}]$ and $P(\rho)$ is increasing on $[0,\infty)$. Substituting $\eqref{z3.52}$ into $\eqref{z3.51}$, we obtain that
\begin{equation*}
    \ln\bigg(\frac{7}{4}\hat{\rho}\bigg)-\ln\bigg(\frac{3}{2}\hat{\rho}\bigg)
    \leq\frac{1}{(2\mu+\lambda)^{\frac{5}{6}}}
    \exp\bigg\{(2+M)^{\mathrm{e}^{\frac{19}{4}D_{2}(1+C_{0})^{5}}}\bigg\},
\end{equation*}
which is impossible if $\lambda$ satisfies \eqref{lam} with $D\geq 5D_2$.
Therefore, we conclude that there is no time $t_1$ such that $\rho(\mathbf{x}(t_1), t_1) = \frac{7}{4}\hat{\rho}$. Since $\mathbf{y}\in\mathbb{R}^2_+$ is arbitrary, it follows that $\rho<\frac{7}{4}\hat{\rho}$ on $\mathbb{R}^2_+\times[0,T]$.

\textbf{Case 2:} If $t_1>1$, one gets from $\eqref{a1}_1$ and \eqref{z1.5} that
\begin{equation*}
    \frac{\mathrm{d}}{\mathrm{d}t}(\rho(t)-\tilde{\rho})+\frac{1}{2\mu+\lambda}\rho(t)(P(\rho(t))-P(\tilde{\rho}))=-\frac{1}{2\mu+\lambda}\rho(t)F(\mathbf{x}(t),t).
\end{equation*}
Multiplying the above equality by $|\rho(t)-\tilde{\rho}|(\rho(t)-\tilde{\rho})$, we have
\begin{equation}\label{z3.53}
    \frac{1}{3}\frac{\mathrm{d}}{\mathrm{d}t}\left|\rho(t)-\tilde{\rho}\right|^3+\frac{1}{2\mu+\lambda}\theta(t)\rho(t)|\rho(t)-\tilde{\rho}|^3
    =-\frac{1}{2\mu+\lambda}\rho(t)(\rho(t)-\tilde{\rho})|\rho(t)-\tilde{\rho}|F(\mathbf{x}(t),t),
\end{equation}
where
\begin{equation*}
  \theta(t)\triangleq\frac{P(\rho(t))-P(\tilde{\rho})}{\rho(t)-\tilde{\rho}}.
\end{equation*}
If $\rho(t)$ takes values in $[\frac{3}{2}\hat{\rho},\frac{7}{4}\hat{\rho}]$,
applying the mean value theorem to $\theta(t)$ and then integrating \eqref{z3.53} from $t_0$ to $t_1$, we obtain from Lemma $\ref{GN}$ and Young's inequality that
\begin{align*}
\hat{\rho}^{3}
&\leq\frac{C}{2\mu+\lambda}\int_{0}^{1}\|F(\cdot,t)\|_{L^\infty}\mathrm{d}t+
\frac{C}{2\mu+\lambda}\int_{1}^{T}\|F(\cdot,t)\|_{L^\infty}^3\mathrm{d}t\notag\\
&\leq\frac{1}{(2\mu+\lambda)^{\frac{5}{6}}}
\exp\bigg\{(2+M)^{\mathrm{e}^{\frac{9}{2}D_{2}(1+C_{0})^{5}}}\bigg\}
+\frac{C}{2\mu+\lambda}\int_{1}^{T}\|F(\cdot,t)\|_{L^{2}}\|\nabla F(\cdot,t)\|_{L^{4}}^{2}\mathrm{d}t\notag\\
&\leq\frac{1}{(2\mu+\lambda)^{\frac{5}{6}}}
\exp\bigg\{(2+M)^{\mathrm{e}^{\frac{9}{2}D_{2}(1+C_{0})^{5}}}\bigg\}
+\frac{(2+M)^{\mathrm{e}^{2D_2(1+C_0)^5}}}{(2\mu+\lambda)^{\frac{1}{2}}}\int_{0}^{T}
\|\nabla\dot{\mathbf{u}}\|_{L^2}\Big(\|\sqrt{\rho}\dot{\mathbf{u}}\|_{L^2}
+C_0^\frac12\|\nabla\dot{\mathbf{u}}\|_{L^2}\Big)\mathrm{d}t\notag\\
&\leq\frac{1}{(2\mu+\lambda)^{\frac{5}{6}}}
\exp\bigg\{(2+M)^{\mathrm{e}^{\frac{9}{2}D_{2}(1+C_{0})^{5}}}\bigg\}
+\frac{1}{(2\mu+\lambda)^{\frac{1}{2}}}(2+M)^{\mathrm{e}^{2D_2(1+C_0)^5}}
\exp\bigg\{(2+M)^{\mathrm{e}^{\frac{9}{2}D_{2}(1+C_{0})^{5}}}\bigg\}\notag\\
&\leq\frac{1}{(2\mu+\lambda)^{\frac12}}
\exp\bigg\{(2+M)^{\mathrm{e}^{\frac{19}{4}D_{2}(1+C_{0})^{5}}}\bigg\},
\end{align*}
which is impossible if $\lambda$ satisfies \eqref{lam} with $D\geq 5D_2$.
Hence we conclude that there is no time $t_1$ such that $\rho(\mathbf{x}(t_1), t_1) = \frac{7}{4}\hat{\rho}$. Since $\mathbf{y}\in\mathbb{R}^2_+$ is arbitrary, it follows that $\rho<\frac{7}{4}\hat{\rho}$ on $\mathbb{R}^2_+\times[0,T]$.
\end{proof}

Now we are ready to prove Proposition $\ref{p3.1}$.
\begin{proof}[Proof of Proposition \ref{p3.1}.]
Proposition \ref{p3.1} follows from Lemmas $\ref{l3.2}$--$\ref{l3.5}$ provided that $\lambda$ satisfies \eqref{lam} with $D\geq 5D_2$.
\end{proof}

\section{Proof of Theorem \ref{t1.1}}\label{sec4}
In this section we apply the \textit{a priori} estimates obtained in Section $\ref{sec3}$ to complete proof of Theorem \ref{t1.1}.
\begin{proof}[Proof of Theorem \ref{t1.1}.]
Let $(\rho_0, \mathbf{u}_0)$ be initial data as described in the theorem.
For $\varepsilon>0$, let $j_\varepsilon=j_\varepsilon(\mathbf{x})$ be the standard mollifier,
define the approximate initial data $(\rho_0^\varepsilon, \mathbf{u}_0^\varepsilon)$:
\begin{align*}
\rho_0^\varepsilon&=[J_\varepsilon\ast(\rho_0\mathbf{1}_{\mathbb{R}^2_+})]
\mathbf{1}_{\mathbb{R}^2_+}+\varepsilon,
\end{align*}
and $\mathbf{u}_0^\varepsilon$ is the unique smooth solution to the following elliptic equation
\begin{equation*}
\begin{cases}
\Delta \mathbf{u}_0^\varepsilon = \Delta(J_\varepsilon\ast\mathbf{u}_0), & \mathbf{x}\in \mathbb{R}^2_+, \\
(u_0^{\varepsilon,1},u_0^{\varepsilon,2})=(\partial_2u_0^{\varepsilon,1},0), & x_2=0.
\end{cases}
\end{equation*}
Then we have
\begin{align*}
(\rho_0^\varepsilon-\tilde{\rho})\in H^2,\ \
\inf_{\mathbf{x}\in\mathbb{R}^2_+}\{\rho_0^\varepsilon(\mathbf{x})\}\geq\varepsilon, \ \ \mathbf{u}_0^\varepsilon\in H^2\cap \widetilde{H}^1.
\end{align*}
By Proposition $\ref{p3.1}$, we see that, for $\varepsilon$ being suitably small,
\begin{align*}
0\leq\rho^\varepsilon(\mathbf{x},t)\leq\frac{7}{4}\hat{\rho}~\textit{a.e.}~\mathrm{on}~\mathbb{R}^2_+\times[0,T]
\end{align*}
provided that $\lambda$ satisfies \eqref{lam}. Thus, Lemma $\ref{l2.1}$ implies the global existence and uniqueness of strong solutions $(\rho^\varepsilon,\mathbf{u}^\varepsilon)$ to \eqref{a1} and \eqref{a3}--\eqref{a4} with the initial data $(\rho_0^\varepsilon,\mathbf{u}_0^\varepsilon)$.

Fix $\mathbf{x}\in\overline{\mathbb{R}^2_+}$ and let $B_R$ be a ball of radius $R$ centered at $\mathbf{x}$. Then, for $t\geq\tau>0$, one gets from Lemmas $\ref{E0}$, $\ref{l3.1}\text{--}\ref{l3.4}$, and Sobolev's inequality that
\begin{align}
\langle\mathbf{u}^\varepsilon(\cdot,t)
\rangle^{\frac12}_{\overline{\mathbb{R}^2_+}}&\leq C\big(1+\|\nabla\mathbf{u}^\varepsilon\|_{L^4}\big)\notag\\
&\leq C\|\nabla\mathbf{u}^\varepsilon\|_{L^2}^{\frac12}
\|\sqrt{\rho^\varepsilon}\dot{\mathbf{u}}^\varepsilon\|_{L^2}
^{\frac12}+\frac{C}{2\mu+\lambda}
\Big(\|\sqrt{\rho^\varepsilon}\dot{\mathbf{u}}^\varepsilon\|_{L^2}^{\frac12}+\|\nabla \mathbf{u}^\varepsilon\|_{L^2}^\frac12\Big)
\|P(\rho^\varepsilon)-P(\tilde{\rho}+\varepsilon)\|_{L^2}^\frac{1}{2}\notag\\
&\quad+\frac{C}{2\mu+\lambda}
\|P(\rho^\varepsilon)-P(\tilde{\rho}+\varepsilon)\|_{L^4}+C\|\nabla \mathbf{u}^\varepsilon\|_{L^2}+C\leq C(\tau).\notag
\end{align}
Noting that
\begin{align}
\bigg|\mathbf{u}^\varepsilon(\mathbf{x},t)-\frac{1}{|B_R\cap\mathbb{R}^2_+|}
\int_{B_R\cap\mathbb{R}^2_+}\mathbf{u}^\varepsilon(\mathbf{y},t)\mathrm{d}\mathbf{y}\bigg|
&=\bigg|\frac{1}{|B_R\cap\mathbb{R}^2_+|}
\int_{B_R\cap\mathbb{R}^2_+}\left(\mathbf{u}^\varepsilon(\mathbf{x},t)
-\mathbf{u}^\varepsilon(\mathbf{y},t)\right)\mathrm{d}\mathbf{y}\bigg|\notag\\
&\leq\frac{1}{|B_R\cap\mathbb{R}^2_+|}C(\tau)\int_{B_R\cap\mathbb{R}^2_+}
|\mathbf{x}-\mathbf{y}|^{\frac12}\mathrm{d}\mathbf{y}\leq C(\tau)R^{\frac12},\notag
\end{align}
one deduces that, for $0<\tau\leq t_1<t_2<\infty$,
\begin{align}
|\mathbf{u}^\varepsilon(\mathbf{x},t_2)-\mathbf{u}^\varepsilon(\mathbf{x},t_1)|
&\leq\frac{1}{|B_R\cap\mathbb{R}^2_+|}\int_{t_{1}}^{t_{2}}
\int_{B_R\cap\mathbb{R}^2_+}|\mathbf{u}_{t}^{\varepsilon}(\mathbf{y},t)
|\mathrm{d}\mathbf{y}\mathrm{d}t+C(\tau)R^{\frac12}\notag\\
&\leq CR^{-1}|t_{2}-t_{1}|^{\frac{1}{2}}\left(\int_{t_{1}}^{t_{2}}
\int\left|\mathbf{u}_{t}^{\varepsilon}(\mathbf{y},t)\right|^{2}
\mathrm{d}\mathbf{y}\mathrm{d}t\right)^{\frac{1}{2}}+C(\tau)R^{\frac12}\notag\\
&\leq CR^{-1}|t_{2}-t_{1}|^{\frac{1}{2}}\left(\int_{t_{1}}^{t_{2}}
\int\left(|\dot{\mathbf{u}}^{\varepsilon}|^{2}
+|\mathbf{u}^{\varepsilon}|^{2}|\nabla\mathbf{u}^{\varepsilon}|^{2}\right)
\mathrm{d}\mathbf{y}\mathrm{d}t\right)^{\frac{1}{2}}
+C(\tau)R^{\frac12}\notag\\
&\leq C(\tau)R^{-1}|t_{2}-t_{1}|^{\frac{1}{2}}+C(\tau)R^{\frac12},\notag
\end{align}
due to
\begin{align}
\int_{t_1}^{t_2}\int|\mathbf{u}^\varepsilon|^2
|\nabla\mathbf{u}^\varepsilon|^2\mathrm{d}\mathbf{x}\mathrm{d}t
&\leq C\sup_{t_1\leq t\leq t_2}\|\mathbf{u}^\varepsilon\|_{L^\infty}^2\int_{t_1}^{t_2}
\int|\nabla\mathbf{u}^\varepsilon|^2\mathrm{d}\mathbf{x}\mathrm{d}t\notag\\
&\leq C\sup_{t_1\leq t\leq t_2}\|\mathbf{u}^\varepsilon\|_{L^2}^{\frac{2}{3}}
\|\nabla\mathbf{u}^\varepsilon\|_{L^4}^{\frac{4}{3}}
\int_{t_1}^{t_2}\int|\nabla\mathbf{u}^\varepsilon|^2\mathrm{d}\mathbf{x}\mathrm{d}t\leq C(\tau).\notag
\end{align}
Choosing $R=|t_2-t_1|^{\frac13}$, one sees that
\begin{equation*}
|\mathbf{u}^\varepsilon(\mathbf{x},t_2)-\mathbf{u}^\varepsilon(\mathbf{x},t_1)|
\leq C(\tau)|t_{2}-t_{1}|^{\frac{1}{6}},~~\text{for}~~0<\tau\leq t_1<t_2<\infty,
\end{equation*}
which implies that $\{\mathbf{u}^\varepsilon\}$ is uniformly H\"older continuous away from $t=0$.

For any fixed $\tau$ and $T$ with $0<\tau<T<\infty$, it follows from Ascoli--Arzel\`{a} theorem that there is a subsequence $\varepsilon_k\rightarrow0$ satisfying
\begin{equation}\label{z4.1}
\mathbf{u}^{\varepsilon_k}\rightarrow \mathbf{u}~~\mathrm{uniformly}~\mathrm{on}~\mathrm{compact}~\mathrm{sets} ~\mathrm{in}~\mathbb{R}^2_+\times(0,\infty).
\end{equation}
Moreover, by the standard compactness arguments as in \cite{Hoff05,EF01,PL98}, we can extract a further subsequence $\varepsilon_{k'}\rightarrow0$ satisfying
\begin{equation}\label{z4.2}
  \rho^{\varepsilon_{k'}}-\tilde{\rho}\rightarrow\rho-\tilde{\rho}~~\mathrm{strongly}~\mathrm{in}~L^p(\mathbb{R}^2_+),~~\mathrm{for}~
  \mathrm{any}~p\in[2,\infty)~\mathrm{and}~t\geq0.
\end{equation}
Therefore, passing to the limit of $\varepsilon_{k'}\rightarrow0$, we deduce from \eqref{z4.1} and \eqref{z4.2} that the limit function $(\rho,\mathbf{u})$ is indeed a weak solution to the problem \eqref{a1}--\eqref{a4} in the sense of Definition $\ref{d1.1}$ and satisfies \eqref{reg}.
\end{proof}

\section{Proof of Theorem \ref{t1.2}}\label{sec5}
This section is devoted to the incompressible limit of \eqref{a1}--\eqref{a4} as the bulk viscosity tends to infinity.

\begin{proof}[Proof of Theorem \ref{t1.2}.]
Let $\{(\rho^\lambda,{\bf u}^\lambda)\}$ be the family of solutions to the problem \eqref{a1}--\eqref{a4} obtained in Theorem \ref{t1.1}. Applying \eqref{reg} and performing a similar argument as that in \eqref{z4.1}, then there is a subsequence $\{(\rho^{\lambda_{k}},{\bf u}^{\lambda_{k}})\}$ such that
\begin{gather}
{\bf u}^{\lambda_{k}}\rightarrow {\bf v}
~~\text{uniformly on compact sets  in}~\mathbb R^2_+\times(0,\infty),\notag\\
\rho^{\lambda_{k}}-\tilde\rho\rightarrow \varrho-\tilde\rho\ \    \text{weakly in}\ \ L^p(\mathbb R^2_+),\ \ \text{for any}\ p\in[2,\infty)\ \text{and}\ {t\ge 0},\label{5.1}\\
\rho^{\lambda_{k}}\rightarrow \varrho\ \  \text{weakly* in}\ \ L^\infty(\mathbb R^2_+),\ \ \text{for any}\ {t\ge 0},\notag\\
\divv{\bf u}^{\lambda_{k}}\rightarrow 0~~\text{strongly  in}~L^2(\mathbb R^2_+\times(0,\infty)).\notag
\end{gather}
Hence, we conclude that $\divv{\bf v}=0$ and $(\varrho,{\bf v})$ satisfies \eqref{1.16}--\eqref{1.17} for all $C^1$ test functions $(\phi,\boldsymbol\psi)$ just as in Definition \ref{d1.2},
with $\divv\boldsymbol\psi=0$ on $\mathbb R^2_+\times[0,\infty)$. Moreover, $(\varrho,{\bf v})$ has the following properties:
\begin{equation}\label{5.2}
 0\leq\varrho({\bf x},t)\leq 2 \hat\rho\ \ \text{a.e. on} \
\mathbb R^2_+\times[0,\infty),
\end{equation}
\begin{equation}\label{5.3}
\sup\limits_{t\ge 0}\big(\|\varrho-\tilde\rho\|_{L^2}^2+\|\sqrt{\varrho}{\bf v}\|_{L^2}^2+\|\nabla{\bf v}\|_{L^2}^2+\sigma\|\nabla^2{\bf v}\|_{L^2}^2\big)+\int_0^\infty\big(\mu\|\nabla{\bf v}\|_{L^2}^2+\|\nabla^2{\bf v}\|_{L^2}^2
\big)\mathrm{d}\tau\le C(C_0,M).
\end{equation}

It remains to show \eqref{1.13}, \eqref{1.14}, and \eqref{1.15}. From the mass equation $\eqref{a1}_1$, one sees that
\begin{equation*}
\partial_t(\rho^{\epsilon,\lambda}-\tilde\rho)^2+{\bf u}^{\epsilon,\lambda}\cdot\nabla(\rho^{\epsilon,\lambda}-\tilde\rho)^2
+2\rho^{\epsilon,\lambda}(\rho^{\epsilon,\lambda}-\tilde\rho)\divv{\bf u}^{\epsilon,\lambda}=0.
\end{equation*}
Integrating the above equality over $\mathbb R^2_+\times(0,t)$, we obtain that
\begin{align*}
\big|\|(\rho^{\epsilon,\lambda}-\tilde\rho)(\cdot,t)\|_{L^2}^2
-\|\rho_0^\epsilon-\tilde\rho\|_{L^2}^2\big|
& =\bigg|\int_0^t\int(\rho^{\epsilon,\lambda}-\tilde\rho)^2\divv{\bf u}^{\epsilon,\lambda}\mathrm{d}{\bf x}\mathrm{d}\tau-2\int_0^t\int\rho^{\epsilon,\lambda}(\rho^{\epsilon,\lambda}
-\tilde\rho)\divv{\bf u}^{\epsilon,\lambda}\mathrm{d}{\bf x}\mathrm{d}\tau\bigg|
\\ & \le C\bigg(\int_0^t\|\rho^{\epsilon,\lambda}-\tilde\rho\|_{L^4}^4
\mathrm{d}\tau\bigg)^\frac{1}{2}\bigg(\int_0^t\|\divv{\bf u}^{\epsilon,\lambda}\|_{L^2}^2\mathrm{d}\tau\bigg)^\frac{1}{2}
\\ & \quad +C\sup\limits_{t\ge 0}\|\rho^{\epsilon,\lambda}(\cdot,t)\|_{L^\infty} \bigg(\int_0^t\|\rho^{\epsilon,\lambda}-\tilde\rho\|_{L^2}^2
\mathrm{d}\tau\bigg)^\frac{1}{2}\bigg(\int_0^t\|\divv{\bf u}^{\epsilon,\lambda}\|_{L^2}^2\mathrm{d}\tau\bigg)^\frac{1}{2}
\\ & \le C(t)\lambda^{-\frac{1}{2}}.
\end{align*}
This together with \eqref{z4.2} implies that
\begin{align*}
\big|\|(\rho^{\lambda}-\tilde\rho)(\cdot,t)\|_{L^2}^2
-\|\rho_0-\tilde\rho\|_{L^2}^2\big|
=\lim\limits_{\epsilon_{k}\rightarrow 0}
\big|\|(\rho^{\epsilon_{k},\lambda}-\tilde\rho)(\cdot,t)\|_{L^2}^2
-\|\rho_0^{\epsilon_{k}}-\tilde\rho\|_{L^2}^2\big|
\le C(t)\lambda^{-\frac{1}{2}},
\end{align*}
which yields that
\begin{equation}\label{5.4}
\lim\limits_{\lambda\rightarrow\infty}\|(\rho^\lambda-\tilde\rho)(\cdot,t)\|_{L^2}
=\|\rho_0-\tilde\rho\|_{L^2},\ \ \text{for any}\ t\ge 0.
\end{equation}
Performing a similar argument as that in \eqref{z3.13}--\eqref{z3.14}, one gets that
\begin{align*}
\|\mathbf{v}\|_{H^1} \leq C(1+C_0)^\frac{1}{2}(1+\|\nabla\mathbf{v}\|_{L^2}),
\end{align*}
which along with \eqref{5.3} implies that
\begin{align*}
\int_0^T\int_{\mathbb R^2_+}|{\bf v}|^2\mathrm{d}{\bf x}\mathrm{d}t&\le C\int_0^T(1+C_0)\big(1+\|\nabla\mathbf{v}\|_{L^2}^2\big)\mathrm{d}t\le C(T),
\end{align*}
as the desired \eqref{1.15}.

Next, in view of \eqref{1.15} and \eqref{5.3}, using the mollifier $j_\epsilon$ as test functions in \eqref{1.16}, we see that, for any compact set $K\subset\mathbb R^2_+$,
\begin{equation}\nonumber\partial_t[\varrho]_\epsilon+{\bf v}\cdot\nabla [\varrho]_\epsilon=\divv\left([\varrho]_\epsilon {\bf v}\right)-\divv[\rho{\bf v}]_\epsilon~~ \text{a.e. on}\ K\times(0,\infty),
\end{equation}
and furthermore,
\begin{equation*}
\partial_t\left([\varrho]_\epsilon-\tilde\rho\right)^2
+{\bf v}\cdot\nabla \left([\varrho]_\epsilon-\tilde\rho\right)^2=2\left([\varrho]_\epsilon-\tilde\rho\right)\big(\divv\left([\varrho-\tilde\rho]_\epsilon {\bf v}\right)-\divv\left[(\varrho-\tilde\rho){\bf v}\right]_\epsilon\big)~~ \text{a.e. on}~  K\times(0,\infty).
\end{equation*}
Integrating the above equality over $K\times(0,t)$, we have that
\begin{align}\label{5.5}
&\big|\|([\varrho]_{\epsilon}-\tilde\rho)(\cdot,t)\|_{L^2(K)}^2- \|[\rho_0]_{\epsilon}-\tilde\rho\|_{L^2(K)}^2\big|\notag\\
 & \le C\sup\limits_{t\ge 0}\|[\varrho]_\epsilon-\tilde\rho\|_{L^\infty}\int_0^t
 \|\divv\left([\varrho-\tilde\rho]_\epsilon {\bf v}\right)-\divv\left[(\varrho-\tilde\rho){\bf v}\right]_\epsilon\|_{L^1(K)}\mathrm{d}\tau
\notag \\
& \le C\int_0^t\|\divv\left([\varrho-\tilde\rho]_\epsilon {\bf v}\right)-\divv\left[(\varrho-\tilde\rho){\bf v}\right]_\epsilon\|_{L^1(K)}\mathrm{d}\tau.
\end{align}
According to Lemma \ref{lcom}, it holds that
\begin{equation*}
\|\divv\left([\varrho-\tilde\rho]_\epsilon {\bf v}\right)-\divv\left[(\varrho-\tilde\rho){\bf v}\right]_\epsilon\|_{L^1(K)}\le C(K)\|\varrho-\tilde\rho\|_{L^2(\mathbb R^2_+)}\|{\bf v}\|_{H^{1}(\mathbb R^2_+)}\in L^1(0,T),\ \ \text{for any}\ T>0.
\end{equation*}
This together with Lebesgue's dominated convergence theorem and Lemma \ref{lcom} leads to
\begin{align}\label{5.6}
&\lim\limits_{\epsilon\rightarrow0}\int_0^t
\|\divv\left([\varrho-\tilde\rho]_\epsilon {\bf v}\right)-\divv\left[(\varrho-\tilde\rho){\bf v}\right]_\epsilon\|_{L^1(K)}\mathrm{d}\tau \notag \\
& =\int_0^t\lim\limits_{\epsilon\rightarrow0}\|\divv\left([\varrho-\tilde\rho]_\epsilon {\bf v}\right)-\divv\left[(\varrho-\tilde\rho){\bf v}\right]_\epsilon\|_{L^1(K)}\mathrm{d}\tau=0.
\end{align}
Substituting \eqref{5.6} into \eqref{5.5}, we deduce that
\begin{align}\label{5.7}
\big|\|(\varrho-\tilde\rho)(\cdot,t)\|_{L^2(K)}^2-\|\rho_0-\tilde\rho\|_{L^2(K)}^2\big|= \lim\limits_{\epsilon\rightarrow0} \big|\| ([\varrho]_{\epsilon}-\tilde\rho)(\cdot,t)\|_{L^2(K)}^2- \|[\rho_0]_{\epsilon}-\tilde\rho\|_{L^2(K)}^2\big|=0,
\end{align}
which yields \eqref{1.14}.

Finally, combining \eqref{5.4} and \eqref{5.7}, one has that
\begin{align*}
\lim\limits_{\lambda\rightarrow\infty}\|(\rho^\lambda-\tilde\rho)(\cdot,t)\|_{L^2(K)}
=\|\varrho-\tilde\rho\|_{L^2(K)},\ \ \text{for any compact set} \ K \subset \mathbb{R}^2_+ \text{ and any} \ t\ge 0,
\end{align*}
which along with \eqref{5.1} gives \eqref{1.13}. Consequently, $(\varrho,{\bf v})$ is a global weak solution to the inhomogeneous incompressible Navier--Stokes equations \eqref{a5} in the sense of Definition \ref{d1.2}.
\end{proof}

\section*{Conflict of interests}
The authors declare that they have no conflict of interests.

\section*{Data availability}
No data was used for the research described in the article.


\end{document}